\documentclass[a4paper,12pt]{scrartcl}
\usepackage{amsmath,amssymb,latexsym,amsthm,enumerate}
\usepackage{xcolor}
\usepackage{graphicx}
\usepackage{epstopdf}
\usepackage{changes}
\usepackage[nocompress]{cite}

\usepackage[T1]{fontenc}
\usepackage[utf8]{inputenc}
\usepackage[english]{babel}

\usepackage{hyperref}
\usepackage[capitalise,noabbrev]{cleveref}

\newcommand*{\wh}{\widehat}
\newcommand*{\wt}{\widetilde}
\newcommand*{\ol}{\overline}

\newcommand*{\cA}{\mathcal{A}}

\newcommand*{\cF}{\mathcal{F}}

\newcommand*{\cL}{\mathcal{L}}

\newcommand*{\cY}{\mathcal{Y}}
\newcommand*{\cZ}{\mathcal{Z}}

\newcommand*{\N}{\mathbb{N}}

\newcommand*{\R}{\mathbb{R}}

\DeclareMathOperator*{\essinf}{ess\,inf}

\newcommand{\be}{\begin{eqnarray*}}
\newcommand{\ee}{\end{eqnarray*}}
\newcommand{\ben}{\begin{eqnarray}}
\newcommand{\een}{\end{eqnarray}}
\newcommand{\bi}{\begin{itemize}}
\newcommand{\ei}{\end{itemize}}

\newcommand*{\mD}{\widetilde{D}}

\newcommand*{\aphi}{\varphi}
\newcommand*{\sfi}{H}
\newcommand*{\ssfi}{\ol{H}}
\newcommand*{\ssfiu}{\wt H}
\newcommand*{\ssfihat}{\wh H}

\newcommand*{\fvJ}{J^{fv}}
\newcommand*{\pmJ}{J^{pm}}
\newcommand*{\LQJ}{J}
\newcommand*{\KTJ}{\hat{J}}

\newcommand*{\rcor}{\ol{r}}

\newcommand*{\md}{\mathbf{d}}
\newcommand*{\kl}{\frac{\lambda_s}{\lambda_s+\kappa_s}}
\newcommand*{\Prog}{\mathrm{Prog}}

\newtheorem{theo}{Theorem}[section]

\newtheorem{lemma}[theo]{Lemma}
\newtheorem{propo}[theo]{Proposition}
\newtheorem{corollary}[theo]{Corollary}

\theoremstyle{definition}
\newtheorem{ex}[theo]{Example}

\newtheorem{remark}[theo]{Remark}

\title{Reducing Obizhaeva-Wang type trade execution problems to LQ stochastic control problems}

\author{Julia Ackermann\thanks{Department of Mathematics \& Informatics, University of Wuppertal, Gaußstr.~20, 42119 Wuppertal, Germany.
\emph{Email:} jackermann@uni-wuppertal.de, \emph{Phone:} +49 (0)202 4395238.}
\and Thomas Kruse\thanks{Department of Mathematics \& Informatics, University of Wuppertal, Gaußstr.~20, 42119 Wuppertal, Germany.
\emph{Email:} tkruse@uni-wuppertal.de, \emph{Phone:} +49 (0)202 4395239.}
\and Mikhail Urusov\thanks{Faculty of Mathematics, University of Duisburg-Essen, Thea-Leymann-Str.~9, 45127 Essen, Germany.
\emph{Email:} mikhail.urusov@uni-due.de, \emph{Phone:} +49 (0)201 1837428.
}}

\begin{document}

\maketitle

\begin{abstract}
We start with a stochastic control problem where the control process is of finite variation (possibly with jumps) and acts as integrator both in the state dynamics and in the target functional.
Problems of such type arise in the stream of literature on optimal trade execution pioneered by Obizhaeva and Wang (models with finite resilience).
We consider a general framework where the price impact and the resilience are stochastic processes.
Both are allowed to have diffusive components.
First we continuously extend the problem from processes of finite variation to progressively measurable processes.
Then we reduce the extended problem to a linear quadratic (LQ) stochastic control problem.
Using the well developed theory on LQ problems we describe the solution to the obtained LQ one and trace it back up to the solution to the (extended) initial trade execution problem.
Finally, we illustrate our results by several examples.
Among other things the examples show
the Obizhaeva-Wang model with random (terminal and moving) targets,
the necessity to extend the initial trade execution problem to a reasonably large class of progressively measurable processes (even going beyond semimartingales) and 
the effects of diffusive components in the price impact process and/or in the resilience process.

\smallskip
\emph{Keywords:}
optimal trade execution;
stochastic price impact;
stochastic resilience;
finite variation stochastic control;
continuous extension of cost functional;
progressively measurable execution strategy;
linear quadratic stochastic control;
backward stochastic differential equation.

\smallskip
\emph{2020 MSC:}
Primary: 91G10; 93E20; 60H10.
Secondary: 60G99.
\end{abstract}

\section*{Introduction}

In the literature on optimal trade execution in illiquid financial markets there arise stochastic control problems where the control is a process of finite variation (possibly with jumps) that acts as integrator both in the state dynamics and in the target functional.
For brevity, we use the term \emph{finite variation stochastic control} for such problems.\footnote{Notice that the class of finite variation stochastic control problems contains the class of singular stochastic control problems.}
In contrast, for control problems where the state is driven by a controlled stochastic differential equation (SDE) and the control acts as one of the arguments in that SDE and as one of the arguments in the integrand of the target functional, we use the term \emph{standard stochastic control problems}.

In this article we present a general solution approach to finite variation stochastic control problems that arise in the literature on optimal trade execution. 
We set up a finite variation stochastic control problem 
of the type of the one in Obizhaeva and Wang \cite{obizhaeva2013optimal}
and its extensions like, e.g., Alfonsi and Acevedo \cite{alfonsi2014optimal},
Bank and Fruth \cite{bank2014optimal}, 
Fruth et al.\ \cite{fruth2014optimal} and \cite{fruth2019optimal}. 
We then show how it can be transformed into a standard linear quadratic
(LQ) 
stochastic control problem which can be solved with the help of state-of-the-art techniques from stochastic optimal control theory. 
In the introduction we first describe the finite variation stochastic control problem and showcase its usage in finance, before presenting our solution approach,
summarizing our main contributions 
and embedding our paper into the literature.

\medskip
\textbf{Finite variation stochastic control problem:}
As a starting point we consider in this paper the following stochastic control problem. Let $T>0$ and let $(\Omega,\cF_T,(\cF_t)_{t\in[0,T]}, P)$ be a filtered probability space satisfying the usual conditions. Let $\xi$ be an $\cF_T$-measurable random variable and let $\zeta=(\zeta_s)_{s\in [0,T]}$ be a progressively measurable process both satisfying suitable integrability assumptions (see \eqref{eq:conditionxi} below). 
Further, let $\lambda=(\lambda_s)_{s\in [0,T]}$ be a bounded progressively measurable process. 
Let $\gamma=(\gamma_s)_{s\in [0,T]}$ be a positive It\^o process driven by some Brownian motion
and $R=(R_s)_{s\in [0,T]}$ an It\^o process driven by a (stochastically) correlated Brownian motion
(see \eqref{eq:resilience} and \eqref{eq:priceimpact} below). 
Throughout the introduction we fix $t\in [0,T]$, $x,d\in \R$ and denote by $\cA^{fv}_t(x,d)$ the set of all adapted, c\`adl\`ag, finite variation processes $X=(X_s)_{s\in[t-,T]}$ satisfying $X_{t-}=x$, $X_T=\xi$, and appropriate integrability assumptions (see (A1)--(A3) below). To each $X\in \cA^{fv}_t(x,d)$ we associate a process $D^X=(D^X_s)_{s\in[t-,T]}$ satisfying
\begin{equation}\label{eq:deviationdyndR_intro}
	dD^X_s = -D^X_sdR_s + \gamma_s dX_s, \quad s \in [t,T], \quad 
	D^X_{t-}=d.
\end{equation}
We consider the finite variation stochastic control problem 
of minimizing the  cost functional
\begin{equation}\label{eq:defcostfct1_intro}
	\fvJ_{t}(x,d,X) = E_t\left[ \int_{[t,T]} \left( D^X_{s-} + \frac12 \Delta X_s \gamma_s \right) dX_s + \int_t^T \lambda_s\gamma_s \left(X_s - \zeta_s \right)^2 ds \right]
\end{equation}
over $X\in\cA^{fv}_t(x,d)$,
where $E_t[\cdot]$ is a shorthand notation for $E[\cdot|\cF_t]$.

\medskip
\textbf{Financial interpretation:}
Stochastic control problems with cost functional of the form \eqref{eq:defcostfct1_intro} or a special case thereof play a central role in the scientific literature on optimal trade execution problems (see the literature discussion below). Consider an institutional investor who holds immediately prior to time $t\in [0,T]$ a position $x\in \R$ ($x>0$ meaning a long position
of $x$ shares of a stock
and $x<0$ a short position
of $-x$ shares) of a certain financial asset. 
The investor trades the asset during the period $[t,T]$ in such a way that at each time $s\in [t-,T]$ the position is given by the value $X_s$ of the adapted, c\`adl\`ag, finite variation process $X=(X_s)_{s\in[t-,T]}$ (satisfying $X_{t-}=x$).
More precisely, $X_{s-}$ 
represents 
the position immediately prior to the trade at time $s$, while $X_s$ is the position immediately after that trade. 
The investor's goal is to reach the target position
$$
X_T=\xi
$$
during the course of the trading period $[t,T]$. Note that we allow $\xi$ to be random to incorporate the possibility that the target position is not known at the beginning of trading but only revealed at terminal time $T$. Such situations may for example be faced by airline companies buying on forward markets the kerosene they need in $T$ months. Their precise demand for kerosene at that future time depends on several factors, such as ticket sales and flight schedules, that are not known today but only gradually learned. 

We assume that the market the investor trades in is illiquid, 
implying that the investor's trades impact the asset price. 
To model this effect, we assume (as is typically done in the literature on optimal trade execution) an additive impact on the price. 
This means that the realized price at which the investor trades at time $r\in [t,T]$ consists of an unaffected price $S^0_r$ plus a deviation $D^X_r$ that is caused by the investor's trades during $[t,r]$. 
We assume that the unaffected price process $S^0=(S^0_r)_{r\in [0,T]}$  is a c\`adl\`ag martingale satisfying appropriate integrability conditions.
Then integration by parts and the martingale property of $S^0$ ensure that expected trading costs due to $S^0$ are given by
$$
E_t\left[\int_{[t,T]}S^0_{r}dX_r\right]=E_t\left[\xi S_T^0\right]-xS_t^0.
$$
Thus, these costs do not depend on the investor's trading strategy $X$ and are therefore neglected in the sequel (we refer to Remark 2.2 in \cite{ackermann2020cadlag} for a more detailed discussion in the case $\xi=0$). The deviation process $D^X$ associated to $X$ is given by \eqref{eq:deviationdyndR_intro}. Informally speaking, we see from \eqref{eq:deviationdyndR_intro} that a trade of size $dX_s$ at time $s\in [t,T]$ impacts $D^X$ by $\gamma_s dX_s$. So, the factor $\gamma_s$ determines how strongly the price reacts to trades, and the process $\gamma$ is therefore called the \emph{price impact process}. In particular, the fact that $\gamma$ is nonnegative entails that a buy trade $dX_s>0$ leads to higher prices whereas a sell trade $dX_s<0$ leads to smaller prices. The second component $-D^X_sdR_s$ in the dynamics \eqref{eq:deviationdyndR_intro} describes the behavior of $D^X$ when the investor is not trading. Typically, it is assumed that $R$ is an increasing process such that in the absence of trades $D^X$ is reverting to $0$ with relative rate $dR_s$. Therefore, $R$ is called the \emph{resilience process}. We refer to \cite{ackermann2021negativeresilience} for a discussion of the effects of  ``negative'' resilience, where $R$ might also be decreasing. We highlight that in the present paper we allow $R$ to have a diffusive part. 
In summary, we note that the deviation prior to a trade of the investor at time $s\in [t,T]$ is given by $D^X_{s-}$ whereas it is equal to $D^X_s=D^X_{s-}+\gamma_s \Delta X_s$ afterwards. We take the mean $D^X_{s-}+\frac{1}{2}\gamma_s \Delta X_s$ of these two values 
as the realized price per unit so that the investor's overall trading costs due to $D^X$ amount to $\int_{[t,T]}\left(D^X_{s-}+\frac{1}{2}\gamma_s \Delta X_s\right)dX_s$. This describes the first integral  on the right-hand side of \eqref{eq:defcostfct1_intro}. Under the assumption that $\lambda$ is nonnegative, the second integral $\int_t^T \lambda_s\gamma_s \left(X_s - \zeta_s \right)^2 ds$ can be understood as a risk term that 
penalizes any deviation of the position $X$ from the moving target $\zeta$ in a quadratic way\footnote{The parametrization $\lambda_s\gamma_s$, $s\in [0,T]$, for the weight is chosen out of mathematical convenience since it makes some of the following assumptions and results shorter to state. Likewise, one can use $\tilde \lambda_s$, $s\in [0,T]$, as a weight and replace $\lambda$ by $\tilde \lambda/\gamma$ in the subsequent assumptions and results.}. A possible and natural choice would be $\zeta_s=E_s[\xi]$, $s\in [0,T]$, so that the risk term ensures that any optimal strategy $X$ does not deviate too much from the (expected) target position $\xi$ in the course of the trading period. 

\medskip
\textbf{Solution approach:} 
The overarching goal of this paper is to show that the finite variation stochastic control problem \eqref{eq:defcostfct1_intro} is equivalent to a standard LQ stochastic control problem (see \cref{cor:equivofinfcostfct} and \cref{cor:uoptimaliffXoptimal} below). The derivation of this result is based on the following insights. The first observation is that, in general, the functional  \eqref{eq:defcostfct1_intro} does not admit a minimizer in $\cA^{fv}_t(x,d)$ (see \Cref{exa:ExampleNonexistenceFromLQct} below for a specific example). In \cite{ackermann2020cadlag} the functional \eqref{eq:defcostfct1_intro} was extended to a set of c\`adl\`ag semimartingales $X$ 
and it was shown that its minimum is attained in this set of semimartingales if and only if a certain process that is derived from the solution of an associated backward stochastic differential equation (BSDE) can be represented by a c\`adl\`ag semimartingale (see Theorem 2.4 in \cite{ackermann2020cadlag}). In this work we go even a step further and extend the functional \eqref{eq:defcostfct1_intro} to the set $\cA^{pm}_t(x,d)$ 
of progressively measurable processes $X=(X_s)_{s\in [t-,T]}$ satisfying appropriate integrability conditions (see (A1) below) and the boundary conditions $X_{t-}=x$ and $X_T=\xi$. To do so, we first derive alternative representations of the first integral inside the expectation in \eqref{eq:defcostfct1_intro} and the deviation in \eqref{eq:deviationdyndR_intro} that do not involve $X\in \cA^{fv}_t(x,d)$ as an integrator (see \cref{propo:costfunctionalpart}). It follows that the resulting alternative representation of $\fvJ$ (see \cref{propo:representationcostfct}) is not only well-defined on $ \cA^{fv}_t(x,d)$ but even on $\cA^{pm}_t(x,d)$, and we denote this extended functional by $\pmJ$ (see \cref{sec:prog_meas_strat}). We next introduce a metric on $\cA^{pm}_t(x,d)$ and prove that $\pmJ$ is the unique continuous extension of $\fvJ$ from $\cA^{fv}_t(x,d)$ to $\cA^{pm}_t(x,d)$ (see \cref{thm:contextcostfct}). In particular, it follows that the infimum of $\fvJ$ over $\cA^{fv}_t(x,d)$ and the infimum of $\pmJ$ over $\cA^{pm}_t(x,d)$ coincide.

Next, for a given $X\in \cA^{pm}_t(x,d)$ we identify the process $\ssfi^X_s =  \gamma_s^{-\frac12} D^X_s - \gamma_s^{\frac12} X_s$, $s\in [t,T]$, as a useful tool in our analysis. Despite $X$ and $D^X$ having discontinuous paths in general, the process $\ssfi^X$, which we call the \textit{scaled hidden deviation process}, is always continuous. Moreover, we show that $\ssfi^X$
can be expressed in feedback form as an It\^o process with coefficients that are 
linear in $\gamma^{-\frac12}D^X$ and $\ssfi^X$ (see \cref{lem:scaledhiddendevdyn}). Subsequently, we reinterpret the process $\gamma^{-\frac12}D^X$ as a control process $u$ and $\ssfi^X$ as the associated state process. Since the cost functional $\pmJ$ is quadratic in $\ssfi^X$ and $u=\gamma^{-\frac12}D^X$, we arrive at a standard LQ stochastic control problem (see \eqref{eq:controlledprocdyn} and \eqref{eq:defcostfct2}) whose minimal costs coincide with the infimum of $\pmJ$ over $\cA^{pm}_t(x,d)$ (see \cref{cor:equivofinfcostfct}). Importantly, there is a one-to-one correspondence between square integrable controls $u$ for this standard problem and strategies $X\in \cA^{pm}_t(x,d)$, which allows to recover the minimizer $X^*\in \cA^{pm}_t(x,d)$ of $\pmJ$ from a minimizer $u^*$ of the standard problem and vice versa (see \cref{cor:uoptimaliffXoptimal}). 

We then solve the LQ stochastic control problem in \eqref{eq:controlledprocdyn} and~\eqref{eq:defcostfct2} using techniques provided in the literature on stochastic optimal control theory. More precisely, 
we apply results from 
Kohlmann and Tang \cite{kohlmann2002global}\footnote{We moreover indicate in \Cref{rem:sun} how we could alternatively use results from Sun et al.~\cite{sun2021indefiniteLQ}.} 
to provide conditions that 
guarantee that an optimal control $u^*$ exists (and is unique). 
This optimal control $u^*$ in the LQ problem is characterized by two BSDEs: one is a quadratic BSDE of Riccati type, the other one is linear, however, with unbounded coefficients
(see \Cref{thm:solnbykohlmanntang}).
In \Cref{cor:solnforJpm} we trace everything back and obtain a unique optimal execution strategy in the class of progressively measurable processes in a closed form (in terms of the solutions to the mentioned BSDEs).

\medskip
\textbf{Summary of our contributions:} 
(a) The Obizhaeva-Wang type finite variation stochastic control problem \eqref{eq:deviationdyndR_intro}--\eqref{eq:defcostfct1_intro} is continuously extended to the set $\cA^{pm}_t(x,d)$ of appropriate progressively measurable processes~$X$.

\smallskip
(b) Problem \eqref{eq:deviationdyndR_intro}--\eqref{eq:defcostfct1_intro} is rather general. In particular, it includes the following features:
\begin{itemize}
\item
Presence of random terminal and moving targets $\xi$ and $(\zeta_s)$;
\item
Price impact is a positive It\^o process $(\gamma_s)$;
\item
Resilience\footnote{To expand on this point, it is worth noting that in our current parametrization, only processes $(R_s)$ with dynamics $dR_s=\rho_s\,ds$ without a diffusive component were considered by now in the literature on optimal trade execution in Obizhaeva-Wang type models. 
Moreover, in most papers $\rho$ is assumed to be positive, that is, only the case of an increasing $(R_s)$ was extensively studied previously.} 
is an It\^o process $(R_s)$ acting as an integrator in~\eqref{eq:deviationdyndR_intro}.
\end{itemize}

\smallskip
(c) 
Via introducing the mentioned \emph{scaled hidden deviation process} $(\ssfi^X_s)$ and reinterpreting the process $(\gamma_s^{-\frac12}D^X_s)$ as a control in an (a priori, different) stochastic control problem, the extended to $\cA^{pm}_t(x,d)$ problem is reduced to an explicitly solvable LQ stochastic control problem. 
Thus, a unique optimal execution strategy in $\cA^{pm}_t(x,d)$ is obtained in a closed form (in terms of solutions to two BSDEs).

\medskip
\textbf{Literature discussion:}
Finite variation stochastic control problems arise in the group of literature on optimal trade execution in limit order books with finite resilience.
The pioneering work\footnote{Posted 2005 on SSRN.} Obizhaeva and Wang \cite{obizhaeva2013optimal} models the price impact via a block-shaped limit order book, where the impact decays exponentially at a constant rate.
This embeds into our model via the price impact process $\gamma$ that is a positive constant and the resilience process $(R_s)$ given by $R_s=\rho s$ with some positive constant $\rho>0$.
Alfonsi et al.\ \cite{alfonsi2008constrained} study constrained portfolio liquidation in the Obizhaeva-Wang model.
Subsequent works within this group of literature either extend this framework in different directions or suggest alternative frameworks with similar features.
There is a subgroup of models which include more general limit order book shapes, see Alfonsi et al.\ \cite{alfonsi2010optimal},
Alfonsi and Schied \cite{alfonsi2010boptimal},
Predoiu et al.\ \cite{predoiu2011optimal}.
Models in another subgroup extend the exponential decay of the price impact to general decay kernels, see Alfonsi et al.\ \cite{alfonsi2012order},
Gatheral et al.\ \cite{gatheral2012transient}.
Models with multiplicative price impact are analyzed in Becherer et al.\ \cite{becherer2018optimala,becherer2018optimalb}.
We mention that in \cite{becherer2018optimalb}, the (multiplicative) deviation is of Ornstein-Uhlenbeck type and incorporates a diffusion term (but this is different from our diffusion term that results from a diffusive part in the resilience $R$). 
Superreplication and optimal investment in a block-shaped limit order book model with exponential resilience is discussed in Bank and Dolinsky \cite{BD_AAP_2019,BD_Bern_2020} and in Bank and Vo\ss{} \cite{BV_SIFIN_2019}.

The present paper falls into the subgroup of the literature that studies time-dependent (possibly stochastic) price impact $(\gamma_s)$ and resilience $(R_s)$ 
in generalized Obizhaeva-Wang models. 
In this connection we mention the works
Alfonsi and Acevedo \cite{alfonsi2014optimal},
Bank and Fruth \cite{bank2014optimal},
Fruth et al.\ \cite{fruth2014optimal},
where deterministically varying price impact and resilience are considered.
Fruth et al.\ \cite{fruth2019optimal} allow for stochastically varying price impact (resilience is still deterministic) and study the arising optimization problem over monotone strategies.
Optimal strategies in a discrete-time model with 
stochastically varying resilience and constant price impact are derived in Siu et al.\ \cite{SiuGuoZhuElliott2019}. 
In Ackermann et al.\ \cite{ackermann2020cadlag,ackermann2021negativeresilience,ackermann2020optimal} both price impact and resilience are stochastic.
We now describe the differences from our present paper in more detail.
In \cite{ackermann2020optimal} optimal execution 
is studied in discrete time via dynamic programming. 
In \cite{ackermann2020cadlag} the framework is the closest to the one in this paper. 
Essentially, our current framework 
is the framework from \cite{ackermann2020cadlag} extended by a 
risk term with some moving target $(\zeta_s)$, 
a possibly non-zero (random) terminal target $\xi$, 
and a larger class of resilience processes (in \cite{ackermann2020cadlag}, as in many previous papers, $(R_s)$ is assumed to have the dynamics $dR_s=\rho_s\,ds$, and $(\rho_s)$ is called resilience). 
In \cite{ackermann2021negativeresilience} the framework is similar to the one in \cite{ackermann2020cadlag}, while the aim is to study qualitative effects of ``negative'' resilience (in the sense that $\rho_s\le 0$ with $(\rho_s)$ as in the previous sentence).
Now, to compare the approach in the present paper with the one in \cite{ackermann2020cadlag}, we first recall that in \cite{ackermann2020cadlag} the finite variation stochastic control problem of the type \eqref{eq:deviationdyndR_intro}--\eqref{eq:defcostfct1_intro} is extended to allow for c\`adl\`ag semimartingale trading strategies $X$ and the resulting optimal execution problem over semimartingales is studied.
The approach in \cite{ackermann2020cadlag} is based on \eqref{eq:deviationdyndR_intro}--\eqref{eq:defcostfct1_intro} (extended with some additional terms), but this does not work beyond semimartingales, as $X$ acts as integrator there.
In contrast, our continuous extension needs to employ essentially different ideas 
since we want to consider the set $\cA^{pm}_t(x,d)$ of progressively measurable strategies (in particular, beyond semimartingales).
This extension is indeed necessary to get an optimizer (see the discussion in the end of \Cref{exa:ExampleNonexistenceFromLQct}). 

Especially with regard to our extension result we now mention several papers where, in different models with finite resilience, trading strategies are not restricted to be of finite variation.
The first instance known to us is Lorenz and Schied \cite{lorenz2013drift}, who discuss dependence of optimal trade execution strategies on a drift in the unaffected price.
In order to react to non-martingale trends they allow for c\`adl\`ag semimartingale trading strategies.
G\^arleanu and Pedersen \cite[Section~1.3]{garleanu2016dynamic} allow for strategies of infinite variation in an infinite horizon portfolio optimization problem under market frictions.
Becherer et al.\ \cite{becherer2019stability} prove a continuous extension result for gains of a large investor in the Skorokhod $J_1$ and $M_1$ topologies in the class of predictable strategies with c\`adl\`ag paths.
As discussed in the previous paragraph in more detail, in \cite{ackermann2020cadlag} the strategies are c\`adl\`ag semimartingales.
In Horst and Kivman \cite{HorstKivman2021} c\`adl\`ag semimartingale strategies emerge in the limiting case of vanishing instantaneous impact parameter, where the initial modeling framework is inspired by Graewe and Horst \cite{graewe2017optimal} and Horst and Xia \cite{horst2019multi}.

To complement the preceding discussion from another perspective, we mention Carmona and Webster \cite{carmona2019selffinancing}, who examine high-frequency trading in limit order books in general (not necessarily related with optimal trade execution).
It is very interesting that one of their conclusions is a strong empirical evidence for the infinite variation nature of trading strategies of high-frequency traders.

Finally, let us mention that, in the context of trade execution problems, risk terms with zero moving target have been included, e.g., in Ankirchner et al.\ \cite{ankirchner2014bsdes},
Ankirchner and Kruse \cite{ankirchner2015optimal},
Graewe and Horst \cite{graewe2017optimal}. Inequality terminal constraints have been considered in Dolinsky et al.\ \cite{dolinsky2020note}, and
risk terms with general terminal and moving targets appear in the models of, e.g.,
Bank et al.\ \cite{BankSonerVoss2017},
Bank and Vo\ss{} \cite{bank2018linear},
Horst and Naujokat \cite{HorstNaujokat2014},
Naujokat and Westray \cite{NaujokatWestray2011}. In particular, \cite{ankirchner2015optimal}, \cite{bank2018linear}, and \cite{dolinsky2020note} consider random terminal targets $\xi$ within trade execution models where position paths are required to be absolutely continuous functions of time. This restriction of the set of position paths entails technical difficulties that make these problems challenging to analyze. In particular, existence of admissible paths that satisfy the terminal constraint is far from obvious and can in general only be assured under further conditions on $\xi$. Since in our model position paths are allowed to jump at terminal time we do not face these challenges in our framework.

\medskip
The paper is structured as follows.
\Cref{sec:problemformulation} is devoted to the continuous extension of our initial trade execution problem to the class of progressively measurable strategies.
\Cref{sec:reductiontoLQ} reduces the problem for the progressively measurable strategies to a standard LQ stochastic control problem.
In \Cref{sec:soln_and_ex} we present the solution to the obtained LQ problem and trace it back up to the solution to the (extended to progressively measurable strategies) trade execution problem.
In \Cref{sec:examples} we illustrate our results with several examples.
Finally, \Cref{sec:proofs} contains the proofs together with some auxiliary results necessary for them.

\section{From finite variation to progressively measurable execution strategies}\label{sec:problemformulation}
In this section we first set up the finite variation stochastic control problem (see \Cref{sec:fvstrat}). In \Cref{sec:alt_rep} we then derive alternative representations of the cost functional and the deviation process which do not require the strategies to be of finite variation. We use these results in \Cref{sec:prog_meas_strat} to extend the cost functional to progressively measurable strategies. In \Cref{sec:continextensioncostfct} we show that this is the unique continuous extension. \Cref{sec:hidden_dev} introduces the hidden deviation process as a key tool for the proofs of \Cref{sec:continextensioncostfct}.
All proofs of this section are deferred to \Cref{sec:proofs}.

\subsection{The finite variation stochastic control problem}\label{sec:fvstrat}
Let $T>0$ and $m\in\N$, $m\geq 2$. We fix a filtered probability space $(\Omega,\cF_T,(\cF_s)_{s\in[0,T]}, P)$ 
satisfying the usual conditions and supporting 
an $m$-dimensional Brownian motion $(W^1,\ldots,W^m)^\top$
with respect to the filtration $(\cF_s)$.

We first fix some notation.
For $t\in [0,T]$ conditional expectations with respect to $\cF_t$ are denoted by $E_t[\cdot]$. For $t\in[0,T]$ and a c\`adl\`ag process $X=(X_s)_{s\in[t-,T]}$ a jump at time $s\in[t,T]$ is denoted by $\Delta X_s= X_s-X_{s-}$. 
We follow the convention that, for $t\in[0,T]$, $r \in [t,T]$ and a c\`adl\`ag semimartingale $L=(L_s)_{s\in [t-,T]}$, jumps of the c\`adl\`ag integrator $L$ at time $t$ contribute to integrals of the form $\int_{[t,r]} \ldots dL_s$. 
In contrast, we write $\int_{(t,r]} \ldots dL_s$ when we do not include jumps of $L$ at time $t$ into the integral. The notation $\int_t^r \ldots dL_s$ is sometimes used for continuous integrators $L$. 
For $n \in \N$ and $y \in \R^n$ let $\lVert y \rVert_2 = (\sum_{j=1}^n y_j^2)^{\frac 12}$. 
For every $t\in [0,T]$ we mean by $L^1(\Omega,\cF_t,P)$ the space of all real-valued $\cF_t$-measurable random variables $Y$ such that $\lVert Y \rVert_{L^1} = E[\lvert Y \rvert] < \infty$.
For $t\in[0,T]$, let $\cL_t^2=\cL^2(\Omega \times [t,T],\Prog(\Omega \times [t,T]),dP\times ds|_{[t,T]})$ denote the space of all (equivalence classes of) real-valued progressively measurable processes $u=(u_s)_{s\in[t,T]}$ such that $\lVert u \rVert_{\cL_t^2} = (E[ \int_t^T u_s^2 ds ])^{\frac12} < \infty$.

\medskip

The control problem we are about to set up requires as input the real-valued, $\mathcal F_T$-measurable random variable $\xi$ and the real-valued, progressively measurable processes $\mu=(\mu_s)_{s \in [0,T]}$, $\sigma=(\sigma_s)_{s \in [0,T]}$, $\rho=(\rho_s)_{s \in [0,T]}$, $\eta=(\eta_s)_{s \in [0,T]}$,  $\rcor=(\rcor_s)_{s \in [0,T]}$, $\zeta=(\zeta_s)_{s\in[0,T]}$ and $\lambda=(\lambda_s)_{s\in[0,T]}$. We suppose that $\mu$, $\sigma$, $\rho$, $\eta$ and $\lambda$ are $dP\times ds|_{[0,T]}$-a.e.\ bounded. 
Moreover, we assume 
that $\rcor$ is $[-1,1]$-valued.  We define $W^R=(W^R_s)_{s\in[0,T]}$ by 
$dW^R_s=\rcor_s dW^1_s + \sqrt{1-\rcor_s^2} dW^2_s$, $s \in [0,T]$, $W^R_0=0$ and refer to $\rcor$ as the \textit{correlation process}.
The processes $\rho$ and $\eta$ give rise to
the continuous semimartingale $R=(R_s)_{s\in[0,T]}$ with 
\begin{equation}\label{eq:resilience}
	dR_s = \rho_s ds + \eta_s dW^R_s, \quad s \in [0,T], \quad  R_0=0,
\end{equation}
which is called the \textit{resilience process}. 
We use the processes $\mu$ and $\sigma$ to define the positive continuous semimartingale $\gamma=(\gamma_s)_{s\in [0,T]}$ by
\begin{equation}\label{eq:priceimpact}
	d\gamma_s = \gamma_s (\mu_s ds + \sigma_s dW^1_s), \quad s\in [0,T],
\end{equation}
with deterministic initial value $\gamma_0>0$. We refer to $\gamma$ as the \textit{price impact process}. Finally, we assume that $\xi$ and $\zeta$ satisfy the integrability conditions
\begin{equation}\label{eq:conditionxi}
	E[\gamma_T \xi^2]<\infty \quad \text{and} \quad 
	E\left[ \int_{0}^T \gamma_s \zeta_s^2 ds \right]<\infty .
\end{equation}

\begin{remark}
Note that the components $W^3,\ldots, W^m$ of the Brownian motion are not needed in the dynamics \eqref{eq:resilience} and~\eqref{eq:priceimpact}. We introduce these components already here, as in \Cref{sec:soln_and_ex},
in order to apply the results from the literature on LQ stochastic control,
we restrict the present setting a little by assuming 
that the filtration $(\cF_s)_{s\in[0,T]}$ is generated by $(W^1,\ldots,W^m)^\top$. The components $W^3,\ldots, W^m$ will therefore serve as further sources of randomness, on which the model inputs may depend.
\end{remark}

We next introduce the finite variation strategies that we consider in the sequel. 
Given $t \in[0,T]$ and $d \in\R$ we associate to an adapted, c\`adl\`ag, finite variation process $X=(X_s)_{s\in[t-,T]}$ a process $D^X=(D^X_s)_{s\in[t-,T]}$ defined by 
\begin{equation}\label{eq:deviationdyndR}
	dD^X_s = -D^X_sdR_s + \gamma_s dX_s, \quad s \in [t,T], \quad 
	D^X_{t-}=d.
\end{equation}
If there is no risk of confusion we sometimes simply write $D$ instead of $D^X$ in the sequel. 
For $t \in[0,T]$, $x,d \in\R$ 
we denote by $\cA^{fv}_t(x,d)$ the set of all adapted, c\`adl\`ag, finite variation  processes $X=(X_s)_{s\in[t-,T]}$ satisfying
$X_{t-}=x$, $X_T=\xi$, and
\begin{enumerate}
 	\item[(A1)] \qquad $E\left[ \int_t^T \gamma_s^{-1} (D_s^X)^2 ds \right] < \infty$,
 	\item[(A2)] \qquad $E\left[ \left( \int_t^T (D_s^X)^4 \gamma_s^{-2} \eta_s^2 ds \right)^{\frac12} \right] < \infty$,
 	\item[(A3)] \qquad $E\left[ \left( \int_t^T (D_s^X)^4 \gamma_s^{-2} \sigma_s^2 ds \right)^{\frac12} \right] < \infty$.
\end{enumerate}
Any element $X\in \cA^{fv}_t(x,d)$ is called a \textit{finite variation execution strategy}. The process $D=D^X$ defined via \eqref{eq:deviationdyndR} is called the associated \textit{deviation process}.

For $t\in [0,T]$, $x,d\in\R$, $X \in \cA^{fv}_t(x,d)$ and associated $D^X$, 
the cost functional $\fvJ$ is given by 
\begin{equation}\label{eq:defcostfct1}
	\fvJ_{t}(x,d,X) = E_t\left[ \int_{[t,T]} D^X_{s-} dX_s + \frac12 \int_{[t,T]} \Delta X_s \gamma_s dX_s + \int_t^T \lambda_s\gamma_s \left(X_s - \zeta_s \right)^2 ds \right].
\end{equation}
(see 
the proofs of 
\Cref{propo:representationcostfct} and \Cref{lem:scaledhiddendevdyn} for well-definedness). 
The finite variation stochastic control problem consists of minimizing the cost functional $\fvJ$ over
$X\in\cA^{fv}_t(x,d)$.

\subsection{Alternative representations for the cost functional and the deviation process}\label{sec:alt_rep}

For $t\in[0,T]$ we introduce an auxiliary process  $\nu=(\nu_s)_{s\in[t,T]}$. It is defined to be the solution of 
\begin{equation}\label{eq:defnu}
	d\nu_s = \nu_s d\left(R_s + [R]_s\right), \quad s \in [t,T], \quad \nu_t=1.
\end{equation} 
Observe that the inverse is given by  
\begin{equation}\label{eq:nuinvdyn}
	d\nu_s^{-1} = - \nu_s^{-1} dR_s, \quad s\in[t,T], \quad \nu_t^{-1}=1. 
\end{equation}

	\begin{remark}\label{rem:120323a1}
	Let $t\in [0,T]$, $d \in \R$. 
	With the definition of $\nu$ in~\eqref{eq:defnu}, it holds for all adapted, c\`adl\`ag, finite variation processes $X=(X_s)_{s\in[t-,T]}$ 
	that the solution $D^X=(D^X_s)_{s\in[t-,T]}$ of the linear SDE~\eqref{eq:deviationdyndR} 
	reads
	$D_s^X = \nu_s^{-1} ( d + \int_{[t,s]} \nu_r \gamma_r dX_r )$, $s\in[t,T].$
	\end{remark}

\begin{propo}\label{propo:costfunctionalpart}
	Let $t \in [0,T]$ and $x,d \in \R$. 
	Suppose that $X=(X_s)_{s\in[t-,T]}$ is an adapted, c\`adl\`ag, finite variation process with $X_{t-}=x$ and with associated 
	process $D^X$ defined by \eqref{eq:deviationdyndR}. 
	It then holds that 
	\begin{equation}\label{eq:costfunctionalpart}
		\begin{split}
			& \int_{[t,T]} D^X_{s-}dX_s + \frac12 \int_{[t,T]} \Delta X_s \gamma_s dX_s 
			= \frac12 \left( \gamma_T^{-1} (D^X_T)^2 - \gamma_t^{-1} d^2 - \int_{t}^T (D^X_s)^2\nu_s^2 d\left(\nu^{-2}_s\gamma_s^{-1}\right) \right)
		\end{split}
	\end{equation}
	and 
	\begin{equation}\label{eq:deviationrewritten}
		\begin{split}
			D^X_r 
			& = \gamma_r X_r + \nu_r^{-1} \left( d - \gamma_t x - \int_t^r X_s d(\nu_s\gamma_s)  \right)
			, \quad r \in [t,T].
		\end{split}
	\end{equation}
\end{propo}

As a consequence of \Cref{propo:costfunctionalpart}, and relying on (A1)--(A3), we can rewrite the cost functional $\fvJ$ as follows.\footnote{Analogues of \Cref{propo:representationcostfct} are present in the literature in other related settings; see, e.g., Lemmas 7.4 and~8.6 in \cite{fruth2014optimal} and the proof of Lemma~5.3 in Appendix~B of \cite{fruth2019optimal}.
A small technical point, which might be worth noting, is that we present a somewhat different proof below.
The idea in \cite{fruth2014optimal,fruth2019optimal} is to derive an analogue of~\eqref{eq:costfunctionalpart} by applying the substitution $dX_s=\gamma_s^{-1}(dD^X_s+D^X_sdR_s)$ and then to compute the expectation.
Exactly the same idea would also work in our present setting but it would result in more sustained calculations and, moreover, the right-hand side of~\eqref{eq:costfunctionalpart} would then look rather different (but this would be an equivalent representation, of course).
The reason for this is that the process $R$, hence $D^X$, can have nonvanishing quadratic variation.
Here we, essentially, express everything not through $D^X$ but rather through $\nu D^X$, which has finite variation by \Cref{rem:120323a1} (as $X$ has finite variation here).
This allows to reduce calculations and provides a somewhat more compact form of~\eqref{eq:costfunctionalpart}.}
To shorten notation, we introduce 
the process $\kappa=(\kappa_s)_{s\in[0,T]}$ defined by 
\begin{equation}\label{eq:defkappa}
\kappa_s = \frac12 \big( 2\rho_s+\mu_s-\sigma_s^2-\eta_s^2 - 2\sigma_s \eta_s \rcor_s \big), \quad s \in [0,T].
\end{equation}

\begin{propo}\label{propo:representationcostfct}
	Let $t \in [0,T]$ and $x,d \in \R$.
	Suppose that $X\in\cA_t^{fv}(x,d)$ with associated deviation process $D^X$ defined by \eqref{eq:deviationdyndR}. 
	It then holds that $\fvJ_t(x,d,X)$ in~\eqref{eq:defcostfct1} 
	admits the representation 
	\begin{equation}\label{eq:costfctrewritten}
		\begin{split}
		\fvJ_t(x,d,X) & = 
		\frac12 E_t\!\left[ \gamma_T^{-1} (D_T^X)^2 + \!\int_t^T \! (D^X_s)^2 \gamma_s^{-1} 2\kappa_s ds + \! \int_t^T \!2 \lambda_s \gamma_s \left(X_s - \zeta_s \right)^2 ds\right] 
		\! - \frac{d^2}{2\gamma_t}  \text{ a.s.}
		\end{split} 
	\end{equation}
\end{propo}

\subsection{Progressively measurable execution strategies}\label{sec:prog_meas_strat}

We point out that the right-hand side of \eqref{eq:costfctrewritten} is also well-defined for progressively measurable processes $X$ satisfying an appropriate integrability condition 
and with associated deviation $D$ defined by \eqref{eq:deviationrewritten} for which one assumes (A1). 
This motivates the following extension of the setting from~\Cref{sec:fvstrat}.

For $t \in [0,T]$, $x,d \in \R$ and a
progressively measurable process $X=(X_s)_{s \in [t-,T]}$ such that $\int_t^T X_s^2 ds < \infty$ a.s.\ and $X_{t-}=x$, 
we define 
the process $D^X=(D^X_s)_{s \in [t-,T]}$ by 
\begin{equation}\label{eq:defdeviationpm}
	D^X_s = \gamma_s X_s + \nu_s^{-1} \left( d - \gamma_t x - \int_t^s X_r d(\nu_r\gamma_r)  \right), \quad s \in [t,T], \quad D^X_{t-}=d 
\end{equation}
(recall $\nu$ from \eqref{eq:defnu}).
Notice that the condition $\int_t^T X_s^2 ds<\infty$ a.s.\ ensures that the stochastic integral in~\eqref{eq:defdeviationpm} is well-defined. 
Again, we sometimes write $D$ instead of~$D^X$.
Further, for $t\in[0,T]$, $x,d \in \R$, let $\cA^{pm}_t(x,d)$ be the set of
(equivalence classes of) 
progressively measurable processes $X=(X_s)_{s \in [t-,T]}$ with
$X_{t-}=x$ and $X_T=\xi$ 
that satisfy $\int_t^T X_s^2 ds < \infty$ a.s.\ and such that condition (A1) holds true for $D^X$ defined by~\eqref{eq:defdeviationpm}.
To be precise, we stress that the equivalence classes for $\cA_t^{pm}(x,d)$ are understood with respect to the equivalence relation 
\begin{align}
X^{(1)}\sim X^{(2)}
\text{ means }
&X^{(1)}_.=X^{(2)}_.\;\;dP\times ds\text{-a.e.\ on }\Omega\times[t,T],
\notag\\
&X^{(1)}_{t-}=X^{(2)}_{t-}\,(=x)
\text{ and }
X^{(1)}_{T}=X^{(2)}_{T}\,(=\xi).
\label{eq:250322a1}
\end{align}
Any element $X\in \cA^{pm}_t(x,d)$ is called a \textit{progressively measurable execution strategy}. Again the process $D=D^X$ now defined via \eqref{eq:defdeviationpm} is called the associated \textit{deviation process}. Clearly, we have that $\cA^{fv}_t(x,d) \subseteq \cA^{pm}_t(x,d)$.

Given $t \in [0,T]$, $x,d \in \R$, and $X \in \cA^{pm}_t(x,d)$ with associated $D^X$ (see \eqref{eq:defdeviationpm}), we define the cost functional $\pmJ$ by 
\begin{equation}\label{eq:defcostfctpm}
	\begin{split}
		\pmJ_t(x,d,X) & = 
		\frac12 E_t\!\left[ \gamma_T^{-1} (D_T^X)^2 + \!\int_t^T \! (D^X_s)^2 \gamma_s^{-1} 2\kappa_s ds + \! \int_t^T \! 2 \lambda_s \gamma_s \left(X_s - \zeta_s \right)^2 ds\right] 
		\! - \frac{d^2}{2\gamma_t} .
	\end{split} 
\end{equation}

Observe that we have the following corollary of \Cref{propo:costfunctionalpart} and \Cref{propo:representationcostfct}.

\begin{corollary}\label{cor:costfctonfvequal}
	Let $t \in [0,T]$, $x,d \in\R$, and $X \in \cA^{fv}_t(x,d)$ with associated deviation process $D^X$ given by \eqref{eq:deviationdyndR}. It then holds that $X \in \cA^{pm}_t(x,d)$, that $D^X$ satisfies \eqref{eq:defdeviationpm}, and that  $\fvJ_t(x,d,X)=\pmJ_t(x,d,X)$.
\end{corollary}

\subsection{The hidden deviation process}\label{sec:hidden_dev}

For $t \in [0,T]$, $x,d \in \R$, and $X \in \cA^{pm}_t(x,d)$ with associated deviation process $D^X$, we define $\sfi^X=(\sfi^X_s)_{s \in [t,T]}$ by 
$\sfi^X_s = D^X_s - \gamma_s X_s$, $s\in[t,T]$. 
Observe that if the investor followed a finite variation execution strategy $X\in \cA^{fv}_t(x,d)$ until time $s\in [t,T]$ and then decided to sell $X_{s}$ units of the asset ($X_{s}<0$ means buying) at time $s$, then by \eqref{eq:deviationdyndR} the resulting deviation at time $s$ would equal $D^X_s - \gamma_s X_s$. The value of $\sfi^X_s$ hence represents the hypothetical deviation if the investor decides to close the position at time $s \in [t,T]$.
We therefore call $\sfi^X$ the \textit{hidden deviation process}. 
Despite $X\in \cA^{pm}_t(x,d)$ and $D^X$ in general being discontinuous, the hidden deviation process $\sfi^X$ is always continuous. 
This can be seen from \eqref{eq:defdeviationpm} and the fact that $R$ (hence also $\nu$) and $\gamma$ are continuous. 
In the case of a finite variation execution strategy $X\in \cA^{fv}_t(x,d)$, it holds that $d\sfi^X_s = -D_sdR_s - X_s d\gamma_s$, $s \in [t,T]$. In particular, the infinitesimal change of the hidden deviation is driven by the changes of the resilience process and the price impact process. 

For $t \in [0,T]$, $x,d \in \R$, and $X \in \cA^{pm}_t(x,d)$, we 
furthermore introduce the \emph{scaled hidden deviation}\footnote{From the mathematical viewpoint, the scaled hidden deviation plays an extremely important role in what follows. It is, therefore, instructive to see in what kind of units it is measured. The meaning of $X$ is quantity (of shares), while both $D^X$ and $\gamma$ are measured in \$. Thus, the scaled hidden deviation $\ssfi^X$ is measured in $\sqrt{\$}$.}
$\ssfi^X=(\ssfi^X_s)_{s \in [t,T]}$ defined by
\begin{equation}\label{eq:scaledhiddendevdef}
	\ssfi_s^X = \gamma_s^{-\frac12} \sfi^X_s = \gamma_s^{-\frac12} D^X_s - \gamma_s^{\frac12} X_s, \quad s \in [t,T].
\end{equation}
Also for $\sfi^X$ and $\ssfi^X$ we sometimes simply write $\sfi$ and $\ssfi$, respectively.
Note that, due to \eqref{eq:defdeviationpm}, it holds that  
$\ssfi^X_s = \gamma_s^{-\frac12} \nu_s^{-1} ( d - \gamma_t x - \int_t^s X_r d(\nu_r\gamma_r) )$, $s\in [t,T]$.

We next show that the scaled hidden deviation process satisfies a linear SDE and an $L^2$-bound. Moreover, we derive a representation of $\pmJ$ in terms of the scaled hidden deviation process.
\begin{lemma}\label{lem:scaledhiddendevdyn}
	Let $t \in [0,T]$, $x,d \in \R$, and $X \in \cA^{pm}_t(x,d)$. 
	Then it holds that 
	\begin{equation}\label{eq:dynscaledhiddendev}
		\begin{split}
			d\ssfi^X_s & = \left(\frac12 \left( \mu_s - \frac14 \sigma_s^2 \right) \ssfi_s^X
			- \frac12 \left( 2(\rho_s+\mu_s) - \sigma_s^2 - \sigma_s\eta_s\rcor_s \right)  \gamma_s^{-\frac12} D^X_s \right)ds\\
			&\quad +\left(\frac12 \sigma_s \ssfi^X_s - (\sigma_s + \eta_s \rcor_s) \gamma_s^{-\frac12} D^X_s \right) dW^1_s
			- \eta_s \sqrt{1-\rcor_s^2} \gamma_s^{-\frac12} D^X_s dW_s^2, 
			\quad s \in [t,T],\\
			\ssfi^X_t & = \frac{d}{\sqrt{\gamma_t}} - \sqrt{\gamma_t}x,
		\end{split}
	\end{equation}
	that $E[\sup_{s\in[t,T]} (\ssfi_s^X)^2]<\infty$,
	and that
	\begin{equation}\label{eq:quadr_pmj}
	\begin{split}
		\pmJ_t(x,d,X)
		& = \frac12 E_t\bigg[ \big(\ssfi^X_T + \sqrt{\gamma_T} \xi \big)^2 
		+ \int_t^T 2(\kappa_s+ \lambda_s) \gamma_s^{-1}(D_s^X)^2
		ds \bigg] - \frac{d^2}{2\gamma_t} \\
		& 
		\quad + E_t\bigg[ \int_t^T \left(  \lambda_s \left(\ssfi^X_s + \sqrt{\gamma_s} \zeta_s \right)^2 - 2 \lambda_s \left(\ssfi^X_s + \sqrt{\gamma_s} \zeta_s \right) \gamma_s^{-\frac 12}D_s^X  \right) ds 
		\bigg] .
	\end{split}
	\end{equation}
\end{lemma}

\subsection{Continuous extension of the cost functional}\label{sec:continextensioncostfct}

\Cref{cor:costfctonfvequal} states that for finite variation execution strategies, the cost functionals $\fvJ$ and $\pmJ$ are the same. 
In this subsection we show that $\pmJ$ can be considered as an extension 
of $\fvJ$ to progressively measurable strategies; i.e., 
we introduce a metric $\md$ on $\cA_t^{pm}(x,d)$ and
show that $\pmJ_t(x,d,X)$ is continuous in the strategy $X \in \cA_t^{pm}(x,d)$ (the first part of \Cref{thm:contextcostfct}),
that $\cA^{fv}_t(x,d)$ is dense in $\cA^{pm}_t(x,d)$
(the second part of \Cref{thm:contextcostfct})
and that the metric space $(\cA_t^{pm}(x,d),\md)$ is complete
(the third part of \Cref{thm:contextcostfct}).
The first and the second parts of \Cref{thm:contextcostfct} mean that, under the metric $\md$, $J_t^{pm}(x,d,\cdot)$ is a unique continuous extension of $J_t^{fv}(x,d,\cdot)$ from $\cA_t^{fv}(x,d)$ onto $\cA_t^{pm}(x,d)$.
The third part of \Cref{thm:contextcostfct} means that, under the metric $\md$, $\cA_t^{pm}(x,d)$ is the largest space where such a continuous extension is uniquely determined by $J_t^{fv}(x,d,\cdot)$ on $\cA_t^{fv}(x,d)$.
This is because the completeness of $(\cA_t^{pm}(x,d),\md)$ is equivalent to the following statement:
For any metric space $(\wh\cA_t(x,d),\wh\md)$ containing $\cA_t^{pm}(x,d)$ and such that $\wh\md|_{\cA_t^{pm}(x,d)^2}=\md$, it holds that the set $\cA_t^{pm}(x,d)$ is closed in $\wh\cA_t(x,d)$.

For $t \in [0,T]$, $x,d\in\R$, and $X,Y \in \cA_t^{pm}(x,d)$ with associated deviation processes $D^X$, $D^Y$ defined by \eqref{eq:defdeviationpm}, we define  
\begin{equation}\label{eq:defmetriconpm}
	\md(X,Y) = 
	\left( E\left[ \int_t^T (D_s^X - D_s^Y)^2 \gamma_s^{-1} ds \right] \right)^{\frac12} .
\end{equation}
Identifying any processes that are equal $dP\times ds|_{[t,T]}$-a.e., this indeed is a metric on $\cA_t^{pm}(x,d)$, see \Cref{lem:ismetric}. 

Note that, for fixed $t\in [0,T]$ and $x,d \in \R$, we may consider the cost functional~\eqref{eq:defcostfctpm} as a function 
$\pmJ_t(x,d,\cdot) \colon (\cA_t^{pm}(x,d),\md) \to (L^1(\Omega,\cF_t,P),\lVert \cdot \rVert_{L^1}).$
Indeed, using (A1), \Cref{lem:scaledhiddendevdyn}, \eqref{eq:conditionxi}, and boundedness of the input processes, we see that 
$\pmJ_t(x,d,X) \in L^1(\Omega,\cF_t,P)$ 
for all $X \in \cA_t^{pm}(x,d)$.

\begin{theo}\label{thm:contextcostfct}
	Let $t\in[0,T]$ and $x,d\in\R$. 
	
	(i) 
	Suppose that $X\in\cA^{pm}_t(x,d)$.
	For every sequence $(X^n)_{n\in\N}$ in $\cA^{pm}_t(x,d)$ with
	$\lim_{n \to \infty} \md(X^n,X) = 0$ it holds that
	$\lim_{n\to\infty} \lVert \pmJ_t(x,d,X^{n}) - \pmJ_t(x,d,X) \rVert_{L^1} = 0$.
	
	(ii)
	For any $X\in\cA^{pm}_t(x,d)$
	there exists a sequence $(X^n)_{n\in\N}$ in $\cA^{fv}_t(x,d)$ such that $\lim_{n \to \infty} \md(X^n,X) = 0$.
	In particular, it holds that
\begin{equation}\label{eq:inf_equal}
\essinf_{X\in\cA^{fv}_t(x,d)}\fvJ_t(x,d,X)=
\essinf_{X\in\cA^{pm}_t(x,d)}\pmJ_t(x,d,X).
\end{equation}	
	
	(iii) 
	For any Cauchy sequence $(X^n)_{n\in\N}$ in $(\cA_t^{pm}(x,d),\md)$ there exists some $X^0 \in \cA_t^{pm}(x,d)$ such that $\lim_{n\to\infty} \md(X^n,X^0)=0$.
\end{theo}

In \Cref{cor:solnforJpm} below we provide sufficient conditions that ensure that the infimum on the right-hand side of \eqref{eq:inf_equal} is indeed a minimum.

\section{Reduction to a standard LQ stochastic control problem}\label{sec:reductiontoLQ}

In this section 
we recast the problem of minimizing  $\pmJ$ 
over $X \in \cA_t^{pm}(x,d)$ as a standard LQ stochastic control problem.
All proofs of this section are given in \Cref{sec:proofs}.

\subsection{The first reduction}
Note that \eqref{eq:quadr_pmj} in \Cref{lem:scaledhiddendevdyn} shows that for $t\in [0,T]$, $x,d\in \R$, and $X\in\cA^{pm}_t(x,d)$ the costs $\pmJ_t(x,d,X)$ depend in a quadratic way on $(\ssfi^X,\gamma^{-\frac12}D^X)$. Moreover, \eqref{eq:dynscaledhiddendev} in \Cref{lem:scaledhiddendevdyn} ensures that the dynamics of $\ssfi^X$ depend linearly on $(\ssfi^X,\gamma^{-\frac12}D^X)$. 
These two observations suggest to view the minimization problem of $\pmJ$ over  $X\in\cA^{pm}_t(x,d)$ as a standard LQ stochastic control problem with state process $\ssfi^X$ and control $\gamma^{-\frac12}D^X$. This motivates the following definitions. 
For every $t \in [0,T]$, $x,d \in \R$, and $u \in \cL_t^2$, we consider the state process $\ssfiu^u=(\ssfiu^u_s)_{s\in[t,T]}$ defined by 
\begin{equation}\label{eq:controlledprocdyn}
	\begin{split}
		d\ssfiu^u_s & = \left(\frac12 \left( \mu_s - \frac14 \sigma_s^2 \right) \ssfiu^u_s -\frac12 \left( 2(\rho_s+\mu_s) - \sigma_s^2 - \sigma_s\eta_s\rcor_s \right) u_s \right)ds\\
		&\quad +\left(\frac12 \sigma_s \ssfiu^u_s - (\sigma_s + \eta_s \rcor_s) u_s\right) dW^1_s
		- \eta_s \sqrt{1-\rcor_s^2} u_s dW_s^2, \quad s\in [t,T],\\
		\ssfiu^u_t&=\frac{d}{\sqrt{\gamma_t}}-\sqrt{\gamma_t}x,
	\end{split}
\end{equation}
and 
the cost functional $\LQJ$ defined by 
\begin{equation}\label{eq:defcostfct2}
	\begin{split}
		\LQJ_{t}\left(\frac{d}{\sqrt{\gamma_t}}-\sqrt{\gamma_t}x,u \right) 
		& = \frac12 E_t\bigg[ \big(\ssfiu^u_T + \sqrt{\gamma_T} \xi \big)^2 
		+ \int_t^T 2(\kappa_s+\lambda_s) u_s^2 
		ds \\
		& \quad \quad \quad 
		+ \int_t^T 
		 \left( 2\lambda_s \left(\ssfiu^u_s + \sqrt{\gamma_s} \zeta_s \right)^2 - 4 \lambda_s \left(\ssfiu^u_s + \sqrt{\gamma_s} \zeta_s \right) u_s \right)
		ds 
		\bigg].
	\end{split}
\end{equation}
Once again we sometimes simply write $\ssfiu$ instead of $\ssfiu^u$. The LQ stochastic control problem is to minimize \eqref{eq:defcostfct2} over the set of admissible controls $\cL_t^2$.

It holds that for every progressively measurable execution strategy $X \in \cA_t^{pm}(x,d)$ there exists a control $u \in \cL_t^2$ such that the cost functional $\pmJ$ can be rewritten in terms of $\LQJ$ (and $-\frac{d^2}{2\gamma_t}$). 
In fact, this is achieved by taking $u=\gamma^{-\frac12}D^X$, as outlined in the motivation above. 
We state this as \Cref{lem:givenXpmfindustcostfcteq}.

\begin{lemma}\label{lem:givenXpmfindustcostfcteq}
	Let $t\in[0,T]$ and $x,d\in\R$. 
	Suppose that $X\in\cA^{pm}_t(x,d)$ with associated deviation $D^X$. 
	Define $u=(u_s)_{s \in [t,T]}$ by $u_s=\gamma_s^{-\frac12} D^X_s$, $s \in [t,T]$. 
	It then holds that $u \in \cL_t^2$ and that 
	$\pmJ_t(x,d,X) = 
	\LQJ_t( \frac{d}{\sqrt{\gamma_t}} - \sqrt{\gamma_t} x, u ) - \frac{d^2}{2\gamma_t}$ a.s.
\end{lemma}

On the other hand, we may also start with $u \in \cL_t^2$ and derive a progressively measurable execution strategy $X \in \cA_t^{pm}(x,d)$ such that the expected costs match. 

\begin{lemma}\label{propo:givenugetX}	
	Let $t\in[0,T]$ and $x,d\in\R$. 
	Suppose that $u=(u_s)_{s \in [t,T]} \in \cL_t^2$ and let $\ssfiu^u$ be the associated solution of \eqref{eq:controlledprocdyn}.
	Define $X=(X_s)_{s\in[t-,T]}$ by 
	$X_s=\gamma_s^{-\frac12}(u_s-\ssfiu^u_s)$, $s\in[t,T)$, $X_{t-}=x$, $X_T=\xi$.
	It then holds that $X \in \cA_t^{pm}(x,d)$ and that 
	$\pmJ_t(x,d,X)=\LQJ_t( \frac{d}{\sqrt{\gamma_t}} - \sqrt{\gamma_t} x, u ) - \frac{d^2}{2\gamma_t}$ a.s.
\end{lemma}

\Cref{lem:givenXpmfindustcostfcteq} and \Cref{propo:givenugetX} together with \Cref{thm:contextcostfct} establish the following equivalence of the control problems pertaining to $\fvJ$, $\pmJ$, and $\LQJ$. 

\begin{corollary}\label{cor:equivofinfcostfct}
	For $t\in[0,T]$ and $x,d\in\R$  
	it holds that 
	\begin{equation*}
		\essinf_{X \in \cA^{fv}_t(x,d)} \fvJ_t(x,d,X) = 
		\essinf_{X \in \cA^{pm}_t(x,d)} \pmJ_t(x,d,X) = \essinf_{u \in \cL_t^2} \LQJ_t\left( \frac{d}{\sqrt{\gamma_t}} - \sqrt{\gamma_t} x, u \right) - \frac{d^2}{2\gamma_t} 
		\text{ a.s.}
	\end{equation*}
\end{corollary}

Furthermore, \Cref{lem:givenXpmfindustcostfcteq}, \Cref{propo:givenugetX}, and \Cref{cor:equivofinfcostfct} provide a method to obtain an optimal progressively measurable execution strategy and potentially an optimal finite variation execution strategy from the standard optimal control problem and vice versa.

\begin{corollary}\label{cor:uoptimaliffXoptimal}
	Let $t\in[0,T]$ and $x,d\in\R$.
	
	(i) 
	Suppose that $X^*=(X^*_s)_{s \in [t-,T]} \in \cA_t^{pm}(x,d)$  minimizes $\pmJ$ over $\cA_t^{pm}(x,d)$ and let $D^{X^*}$ be the associated deviation process.
	Then, $u^*=(u^*_s)_{s \in [t,T]}$ defined by $u^*_s=\gamma_s^{-\frac12} D^{X^*}_s$, $s \in [t,T]$, minimizes $\LQJ$ over $\cL_t^2$.
	
	(ii) 
	Suppose that $u^*=(u^*_s)_{s \in [t,T]} \in \cL_t^2$ minimizes $\LQJ$ over $\cL_t^2$ and let $\ssfiu^{u^*}$ be the associated solution of \eqref{eq:controlledprocdyn} for $u^*$.
	Then, $X^*=(X^*_s)_{s\in[t,T]}$ defined by  $X^*_s=\gamma_s^{-\frac12}(u^*_s-\ssfiu^{u^*}_s)$, $s \in [t,T)$, $X_{t-}^*=x$, $X_T^*=\xi$, minimizes $\pmJ$ over $\cA_t^{pm}(x,d)$.
	
	Moreover, if $X^* \in \cA_t^{fv}(x,d)$ (in the sense that there is an element of $\cA_t^{fv}(x,d)$ within the equivalence class of $X^*$, see~\eqref{eq:250322a1}), 
	then $X^*$ minimizes $\fvJ$ over $\cA_t^{fv}(x,d)$.
\end{corollary}

\subsection{Formulation without cross-terms}\label{sec:LQproblemwocrossterms}

Note that the last integral in the definition \eqref{eq:defcostfct2} of the cost functional $J$ involves a product between the state process $\ssfiu^u$ and the control process $u$. A larger part of
the literature on LQ optimal control considers cost functionals that do
not contain such cross-terms. In particular, this applies to \cite{kohlmann2002global}, whose results we apply in \Cref{sec:soln_and_ex} below. 
For this reason we provide in this subsection  a reformulation of the control problem \eqref{eq:controlledprocdyn} and \eqref{eq:defcostfct2}
that does
not contain cross-terms. 
In order to carry out the transformation necessary for this, we need to impose a further condition on our model inputs. 
We assume that there exists a constant $C\in [0,\infty)$ such that for all $s\in [0,T]$ we have $P$-a.s.\ that
\begin{equation}\label{eq:cond_no_ct}
\lvert \lambda_s \rvert \le C\lvert \lambda_s+\kappa_s\rvert.
\end{equation}
Note that this assumption ensures that the set $\{\lambda_s+\kappa_s=0\}$ is a subset of $\{\lambda_s=0\}$ (up to a $P$-null set).
For this reason we, in the sequel, use the following

\smallskip\noindent
\textbf{Convention:}
Under~\eqref{eq:cond_no_ct} we always understand $\frac{\lambda_s}{\lambda_s+\kappa_s}=0$ on the set $\{\lambda_s+\kappa_s=0\}$.

\smallskip
Now in order to get rid of the cross-term in \eqref{eq:defcostfct2} we transform for $t\in [0,T]$ any control process $u  \in \cL_t^2$ in an affine way to $\hat u_s=u_s-\frac{\lambda_s}{\lambda_s+\kappa_s}(\ssfiu^u_s+\sqrt{\gamma_s}\zeta_s)$, $s\in [t,T]$. This leads to the new controlled state process $\ssfihat^{\hat u}=(\ssfihat^{\hat u}_s)_{s\in[t,T]}$ which is defined for every $t\in[0,T]$, $x,d \in \R$, and $\hat u \in \cL_t^2$ by 
\begin{equation}\label{eq:controlledprocdynhat}
	\begin{split}
		d\ssfihat_s^{\hat u}
		& = \left( \frac{\mu_s}{2} - \frac18 \sigma_s^2 - \frac{\lambda_s}{\lambda_s+\kappa_s} \left( \rho_s+\mu_s - \frac{\sigma_s^2 + \sigma_s \eta_s \rcor_s}{2} \right) \right)  \ssfihat_s^{\hat u} ds \\
		& \quad - \left( \rho_s+\mu_s - \frac{\sigma_s^2 + \sigma_s \eta_s \rcor_s}{2} \right) \hat{u}_s ds
		- \frac{\lambda_s}{\lambda_s+\kappa_s} \left( \rho_s+\mu_s - \frac{\sigma_s^2 + \sigma_s \eta_s \rcor_s}{2} \right) \sqrt{\gamma_s} \zeta_s ds \\
		& \quad + \left( \frac{\sigma_s}{2} - \frac{\lambda_s}{\lambda_s+\kappa_s} (\sigma_s + \eta_s \rcor_s) \right) \ssfihat_s^{\hat u} dW_s^1 
		- (\sigma_s + \eta_s \rcor_s) \hat{u}_s dW_s^1 \\
		& \quad - \frac{\lambda_s}{\lambda_s+\kappa_s} (\sigma_s + \eta_s \rcor_s) \sqrt{\gamma_s} \zeta_s dW_s^1 
		- \frac{\lambda_s}{\lambda_s+\kappa_s} \eta_s \sqrt{1-\rcor_s^2} \ssfihat_s^{\hat u} dW_s^2 \\ 
		& \quad - \eta_s \sqrt{1-\rcor_s^2} \hat u_s dW_s^2 
		- \frac{\lambda_s}{\lambda_s+\kappa_s} \eta_s \sqrt{1-\rcor_s^2}  \sqrt{\gamma_s} \zeta_s dW_s^2 
		, \quad s\in [t,T], \\
		\ssfihat_t^{\hat u} & = \frac{d}{\sqrt{\gamma_t}}-\sqrt{\gamma_t}x. 
	\end{split}
\end{equation}
The meaning of~\eqref{eq:controlledprocdynhat} is that we only reparametrize the control ($u\to\hat u$) but not the state variable ($\ssfihat^{\hat u}=\ssfiu^u$), see \Cref{lem:trafouvsuhat} for the formal statement. 
For $t\in [0,T]$, $x,d\in\R$, $\hat u \in \cL_t^2$ and associated $\ssfihat^{\hat u}$, we define the cost functional $\KTJ$ by 
\begin{equation}\label{eq:defcostfctLQ2}
	\begin{split}
		\KTJ_{t}\!\left(\!\frac{d}{\sqrt{\gamma_t}}-\!\sqrt{\gamma_t}x, \hat u \! \right) 
		& \! = \! E_t\bigg[ \frac12 \big(\ssfihat^{\hat u}_T + \! \sqrt{\gamma_T} \xi \big)^2  \!
		+ \!\! \int_t^T \! \!\!
		 \left(\! \frac{\lambda_s \kappa_s}{\lambda_s+\kappa_s} \left(\! \ssfihat^{\hat u}_s + \! \sqrt{\gamma_s} \zeta_s \right)^2 \!\! + \! (\lambda_s\!+\!\kappa_s) \hat u_s^2 \right)
		ds \bigg].
	\end{split}
\end{equation}
This cost functional does not exhibit cross-terms, but is equivalent to $\LQJ$ of~\eqref{eq:defcostfct2} in the sense of the following lemma.

\begin{lemma}\label{lem:trafouvsuhat}
	 Assume that \eqref{eq:cond_no_ct} holds true.
	 Let $t\in [0,T]$ and $x,d\in\R$. 
	
	(i) 
	Suppose that $u \in\cL_t^2$ with associated state process $\ssfiu^u$ defined by \eqref{eq:controlledprocdyn}. 
	Then, $\hat u =(\hat u_s)_{s \in [t,T]}$ defined by 
	$\hat u_s = u_s - \frac{\lambda_s}{\lambda_s+\kappa_s} (\ssfiu^u_s + \sqrt{\gamma_s} \zeta_s )$, $s\in[t,T]$, 
	is in $\cL_t^2$, and it holds that $\ssfihat^{\hat u}=\ssfiu^u$ and 
	$\LQJ_{t}(\frac{d}{\sqrt{\gamma_t}}-\sqrt{\gamma_t}x, u ) = \KTJ_{t}(\frac{d}{\sqrt{\gamma_t}}-\sqrt{\gamma_t}x, \hat u )$.
	
	(ii) 
	Suppose that $\hat u \in \cL_t^2$ with associated state process $\ssfihat^{\hat u}$ defined by \eqref{eq:controlledprocdynhat}. 
	Then, $u=(u_s)_{s \in [t,T]}$ defined by 
	$u_s=\hat u_s + \frac{\lambda_s}{\lambda_s+\kappa_s} (\ssfihat^{\hat u}_s + \sqrt{\gamma_s} \zeta_s )$, $s\in[t,T]$, 
	is in $\cL_t^2$, and it holds that $\ssfiu^u=\ssfihat^{\hat u}$ and 
	$\LQJ_{t}(\frac{d}{\sqrt{\gamma_t}}-\sqrt{\gamma_t}x, u ) = \KTJ_{t}(\frac{d}{\sqrt{\gamma_t}}-\sqrt{\gamma_t}x, \hat u )$.
\end{lemma}

As a corollary, we obtain the following link between an optimal control for $\KTJ$ and an optimal control for $\LQJ$.
\begin{corollary}\label{cor:uoptimaliffuhatoptimal}
	Assume that \eqref{eq:cond_no_ct} holds true.
	Let $t\in[0,T]$ and $x,d\in\R$.
	
	(i) 
	Suppose that $u^*=(u^*_s)_{s \in [t,T]} \in \cL_t^2$ is  an optimal control for $\LQJ$, and let $\ssfiu^{u^*}$ be the solution of \eqref{eq:controlledprocdyn} for $u^*$. 
	Then, $\hat u^*=(\hat u^*_s)_{s \in [t,T]}$ defined by 
	$\hat u^*_s = u^*_s - \frac{\lambda_s}{\lambda_s+\kappa_s} (\ssfiu^{u^*}_s + \sqrt{\gamma_s} \zeta_s )$, $s\in[t,T]$, 
	is an optimal control in $\cL_t^2$ for $\KTJ$.
	
	(ii) 
	Suppose that $\hat u^*=(\hat u^*_s)_{s \in [t,T]} \in \cL_t^2$ is an optimal control for $\KTJ$, and let $\ssfihat^{\hat u^*}$ be the solution of \eqref{eq:controlledprocdynhat} for $\hat u^*$.
	Then, $u^*=(u^*_s)_{s \in [t,T]}$ defined by 
	$u^*_s=\hat u^*_s + \frac{\lambda_s}{\lambda_s+\kappa_s} (\ssfihat^{\hat u^*}_s + \sqrt{\gamma_s} \zeta_s )$, $s\in[t,T]$, 
	is an optimal control in $\cL_t^2$ for $\LQJ$.
\end{corollary}

\section{Solving the LQ control problem and the trade execution problem}\label{sec:soln_and_ex}

We now solve the LQ control problem from \Cref{sec:reductiontoLQ} and consequently obtain a solution of the trade execution problem.

\begin{remark}\label{rem:sun}
The solution approach of \cite{kohlmann2002global}, which we are about to apply, is built on the tight connection between standard LQ stochastic control problems and Riccati-type BSDEs (BSRDEs). This connection is well known and dates back at least to Bismut (see, e.g., \cite{bismut1976linear} and \cite{bismut1978controle}). The central challenge in this approach is to establish the existence of a solution of the BSRDE. Kohlmann and Tang prove in \cite{kohlmann2002global} such results in a general framework which in particular covers our problem formulation in \Cref{sec:LQproblemwocrossterms} under appropriate assumptions. 

There is a variety of further results in the literature on LQ stochastic control problems that provide existence results for BSRDEs under different sets of assumptions.
A specific potential further possibility is, for example, to use the results of the recent article \cite{sun2021indefiniteLQ} by Sun et al.\ in our setting. The set-up of \cite{sun2021indefiniteLQ} allows for cross-terms in the cost functional and, more interestingly, the results in \cite{sun2021indefiniteLQ} hold under a uniform convexity assumption on the cost functional, which is a weaker requirement than the usually imposed nonnegativity and positivity assumptions on the coefficients of the cost functional. 
However, in general, the terminal costs and the running costs in \eqref{eq:defcostfct2} (and also in \eqref{eq:defcostfctLQ2}) contain terms 
	such as 
$(\ssfiu^u_T + \sqrt{\gamma_T} \xi )^2$ and $\lambda_s(\ssfiu^u_s + \sqrt{\gamma_s} \zeta_s )^2$, which are inhomogeneous. 
Therefore, the results of \cite{sun2021indefiniteLQ} are only directly applicable in the special case where $\xi=0$ and at least one of $\lambda$ and $\zeta$ vanishes. 
A possible route for future research could be to incorporate inhomogeneous control problems as presented in  \Cref{sec:reductiontoLQ} to the framework of \cite{sun2021indefiniteLQ}.
\end{remark}

\textbf{Setting in \Cref{sec:soln_and_ex}:}
In our general setting (see \Cref{sec:fvstrat}) we additionally assume 
that the filtration $(\cF_s)_{s\in[0,T]}$ for the filtered probability space $(\Omega,\cF_T,(\cF_s)_{s\in[0,T]}, P)$ 
is the augmented natural filtration of the Brownian motion $(W^1,\ldots,W^m)^\top$. 
Furthermore, we set the initial time to $t=0$. 
We also assume that $\lambda$ and $\kappa = \frac12 ( 2\rho+\mu-\sigma^2-\eta^2 - 2\sigma \eta \rcor )$ are nonnegative $dP\times ds|_{[0,T]}$-a.e.\footnote{We stress at this point that the results presented in \Cref{sec:problemformulation,sec:reductiontoLQ} are valid for more general filtrations and for processes $\lambda$ and $\kappa$ possibly taking negative values. This opens the way for applying \Cref{sec:problemformulation,sec:reductiontoLQ} in other settings in future research.}

\begin{remark}
Note that the assumption of nonnegativity of $\lambda$ and $\kappa$ is necessary to apply the results of \cite{kohlmann2002global}. Indeed, \cite{kohlmann2002global} requires that $\lambda+\kappa$ (the coefficient in front of $\hat u^2$ in~\eqref{eq:defcostfctLQ2}) and $\frac{\lambda \kappa}{\lambda +\kappa}$ (the coefficient in front of $(\ssfihat^{\hat u}_s + \sqrt{\gamma_s} \zeta_s )^2$ in~\eqref{eq:defcostfctLQ2}) are nonnegative and bounded, which implies that $\lambda$ and $\kappa$ have to be nonnegative.

Moreover, we note that nonnegativity of $\lambda$ and $\kappa$ ensures that \eqref{eq:cond_no_ct} is satisfied. 
Further, we observe that the mentioned coefficients $\lambda+\kappa$ and $\frac{\lambda \kappa}{\lambda +\kappa}$ are bounded, as required. 
Indeed, it clearly holds $\frac{\lambda \kappa}{\lambda +\kappa}\le\kappa$, and it remains to recall that $\mu, \sigma, \rho, \eta$, and $\lambda$ are bounded and $\rcor$ is $[-1,1]$-valued (see \Cref{sec:fvstrat}).
\end{remark}

Note that the LQ control problem of \Cref{sec:LQproblemwocrossterms}, which consists of minimizing $\KTJ$ in \eqref{eq:defcostfctLQ2} with state dynamics given by \eqref{eq:controlledprocdynhat}, is of the form considered in \cite[(79)-(81)]{kohlmann2002global}. The solution can be described by the two BSDEs \cite[(9) and (85)]{kohlmann2002global}. The first one, \cite[(9)]{kohlmann2002global}, is a Riccati-type BSDE, which in our setting reads
\begin{equation}\label{eq:BSDEKL}
	\begin{split}
		dK_s & = - \Bigg[  \left( \mu_s + \kl \bigg( \frac{\lambda_s}{\lambda_s+\kappa_s}(\sigma_s^2 + 2\sigma_s \eta_s \rcor_s + \eta_s^2) - 2(\rho_s + \mu_s)  \bigg) \right)  K_s \\
		& \qquad + \left( \sigma_s - \kl 2(\sigma_s + \eta_s \rcor_s) \right) L^1_s 
		- \kl 2 \eta_s \sqrt{1-\rcor_s^2} L^2_s 
		+ \frac{\lambda_s\kappa_s}{\lambda_s+\kappa_s} 
		\\[1mm]
		& \qquad - 
		\frac{
		\left( \left( \rho_s + \mu_s - \kl (\sigma_s^2 + 2\sigma_s \eta_s \rcor_s + \eta_s^2) \right) K_s + (\sigma_s + \eta_s \rcor_s)L^1_s + \eta_s \sqrt{1-\rcor_s^2} L^2_s \right)^2}{
		\lambda_s + \kappa_s + (\sigma_s^2 + 2\sigma_s \eta_s \rcor_s + \eta_s^2 ) K_s} 
		\Bigg] ds \\
		& \quad + \sum_{j=1}^m L_s^j dW_s^j , \quad s \in [0,T], \\
		K_T&=\frac12 .
	\end{split}
\end{equation}
We call a pair $(K,L)$ with $L=(L^1,L^2,\ldots,L^m)^\top$ a \emph{solution} to BSDE~\eqref{eq:BSDEKL} if 
\begin{enumerate}[(i)]
	\item
	$K$ is an adapted, continuous, nonnegative, and bounded process,
	
	\item\label{it:010422a1}
	$\lambda + \kappa + (\sigma^2 + 2\sigma \eta \rcor + \eta^2 ) K >0$ $dP\times ds|_{[0,T]}$-a.e.,
	
	\item
	$L^1,\ldots, L^m\in \cL_0^2$, and
	
	\item
	BSDE~\eqref{eq:BSDEKL} is satisfied $P$-a.s.
\end{enumerate}
A discussion of this definition is in order.
The requirement of nonnegativity and boundedness of $K$ can be explained at this point by the fact that, under mild conditions, such a solution exists (see \Cref{thm:solnbykohlmanntang} below).
Condition~\eqref{it:010422a1} ensures that there is no problem with division in the driver of~\eqref{eq:BSDEKL}, where the quantity $\lambda + \kappa + (\sigma^2 + 2\sigma \eta \rcor + \eta^2 ) K$ appears in the denominator.
Moreover, it is worth noting that, for a nonnegative $K$, in our setting we always have
$\lambda + \kappa + (\sigma^2 + 2\sigma \eta \rcor + \eta^2 ) K\ge0$,
as $\sigma^2 + 2\sigma \eta \rcor + \eta^2=(\sigma+\eta\rcor)^2+\eta^2(1-\rcor^2)$.
From this we also see that the quantity $\lambda + \kappa + (\sigma^2 + 2\sigma \eta \rcor + \eta^2 ) K$ can vanish only in ``very degenerate'' situations.
The conclusion is that condition~\eqref{it:010422a1} is quite natural.

To shorten notation, we introduce, for a solution $(K,L)$ of BSDE~\eqref{eq:BSDEKL}, the process $\theta=(\theta_s)_{s\in[0,T]}$ by,
for $s\in[0,T]$,
\begin{equation}\label{eq:theta}
	\theta_s = \frac{\left( \rho_s + \mu_s - \kl (\sigma_s^2 + 2\sigma_s \eta_s \rcor_s + \eta_s^2) \right) K_s + (\sigma_s + \eta_s \rcor_s)L^1_s + \eta_s \sqrt{1-\rcor_s^2} L^2_s}{\lambda_s + \kappa_s + (\sigma_s^2 + 2\sigma_s \eta_s \rcor_s + \eta_s^2 ) K_s}.
\end{equation}
Next, we consider the second BSDE \cite[(85)]{kohlmann2002global}, which is linear and reads in our setting 
\begin{equation}\label{eq:BSDEpsiphi}
	\begin{split}
		d\psi_s & = - \Bigg[ 
		\left( \frac{\mu_s}{2} - \frac{\sigma_s^2}{8} 
		-  \left(\rho_s+\mu_s-\frac{\sigma_s^2 + \sigma_s\eta_s\rcor_s}{2}\right) 
		\left( \kl + \theta_s \right) \right) \psi_s \\
		& \qquad + \left( \frac{\sigma_s}{2} - (\sigma_s + \eta_s \rcor_s) \left( \kl + \theta_s \right) \right) 
		 \left( \phi_s^1 + \kl (\sigma_s + \eta_s \rcor_s) \sqrt{\gamma_s} \zeta_s K_s \right) \\
		& \qquad - \eta_s \sqrt{1-\rcor_s^2} \left( \kl + \theta_s \right) 
		 \left( \phi_s^2 + \kl \eta_s \sqrt{1-\rcor_s^2} \sqrt{\gamma_s} \zeta_s K_s \right) \\
		& \qquad + \kl \sqrt{\gamma_s} \zeta_s 
		\left( \left( \rho_s + \mu_s - \frac{\sigma_s^2+\sigma_s\eta_s\rcor_s}{2}  \right) K_s 
		+ (\sigma_s + \eta_s\rcor_s) L_s^1 
		+ \eta_s \sqrt{1-\rcor_s^2} L_s^2 \right) \\
		& \qquad - \frac{\lambda_s\kappa_s}{\lambda_s+\kappa_s}  \sqrt{\gamma_s} \zeta_s 
		\Bigg] ds  
		+ \sum_{j=1}^m \phi_s^j dW_s^j, \quad s\in [0,T], \\
		\psi_T & = -\frac12 \sqrt{\gamma_T} \xi .
	\end{split}
\end{equation}
A pair $(\psi,\phi)$ with $\phi=(\phi^1,\phi^2,\ldots,\phi^m)^\top$ is called a \emph{solution} to BSDE~\eqref{eq:BSDEpsiphi} if 
\begin{enumerate}[(i)]
	\item
	$\psi$ is an adapted continuous process with $E\left[ \sup_{s\in[0,T]} \psi_s^2 \right]<\infty$,
	
	\item
	$\phi$ is progressively measurable with $\int_0^T \lVert \phi_s \rVert_2^2 ds < \infty$ $P$-a.s., and

	\item
	BSDE~\eqref{eq:BSDEpsiphi} is satisfied $P$-a.s.
\end{enumerate}

For a solution $(K,L)$ of BSDE~\eqref{eq:BSDEKL} and a corresponding solution $(\psi,\phi)$ of BSDE~\eqref{eq:BSDEpsiphi}, we define $\theta^0=(\theta^0_s)_{s\in[0,T]}$ by
\begin{equation}\label{eq:theta0}
	\begin{split}
	\theta_s^0 
	& = \Bigg( \bigg( \rho_s + \mu_s - \frac{\sigma_s^2+\sigma_s\eta_s\rcor_s}{2} \bigg) \psi_s
	+ \kl \sqrt{\gamma_s} \zeta_s (\sigma_s^2 + 2\sigma_s \eta_s \rcor_s + \eta_s^2) K_s  \\
	& \qquad 
	+ (\sigma_s + \eta_s \rcor_s) \phi^1_s 
	+ \eta_s \sqrt{1-\rcor_s^2} \phi^2_s \Bigg) 
	\cdot \left(\lambda_s + \kappa_s + (\sigma_s^2 + 2\sigma_s \eta_s \rcor_s + \eta_s^2 ) K_s\right)^{-1},
	\end{split}
\end{equation}
for $s\in[0,T]$.
We further introduce for $x,d\in\R$ and $s \in [0,T]$
the SDE
\begin{equation}\label{eq:optssfihat}
d\ssfihat_s^*=\ssfihat_s^*\,d\cY_s+d\cZ_s,\quad
\ssfihat_0^* = \frac{d}{\sqrt{\gamma_0}} - \sqrt{\gamma_0} x,
\end{equation}
where for $s \in [0,T]$
\begin{equation*}
\begin{split}
d\cY_s&=\left( \frac{\mu_s}{2} - \frac{\sigma_s^2}{8} - \bigg( \rho_s + \mu_s - \frac{\sigma_s^2+\sigma_s\eta_s\rcor_s}{2} \bigg) \bigg( \kl + \theta_s \bigg) \right) ds\\
& \quad + \left( \frac{\sigma_s}{2} - (\sigma_s+\eta_s\rcor_s) \bigg( \kl + \theta_s \bigg) \right) dW_s^1
- \eta_s \sqrt{1-\rcor_s^2} \bigg( \kl + \theta_s \bigg) dW_s^2 ,
\end{split}
\end{equation*}
\begin{equation*}
\begin{split}
d\cZ_s&= \bigg( \rho_s + \mu_s - \frac{\sigma_s^2+\sigma_s\eta_s\rcor_s}{2} \bigg) \bigg( \theta_s^0 - \sqrt{\gamma_s} \zeta_s \kl \bigg) ds \\
& \quad + (\sigma_s + \eta_s\rcor_s) \bigg( \theta_s^0 - \sqrt{\gamma_s} \zeta_s \kl \bigg) dW_s^1 
+ \eta_s \sqrt{1-\rcor_s^2} \bigg( \theta_s^0 - \sqrt{\gamma_s} \zeta_s \kl \bigg) dW_s^2
.
\end{split}
\end{equation*}
We will show that the solution $\ssfihat^*$ of~\eqref{eq:optssfihat} is the optimal state process in the stochastic control problem to minimize $\KTJ$ of~\eqref{eq:defcostfctLQ2}.
Notice that $\ssfihat^*$ can be easily expressed via $\cY$ and $\cZ$ in closed form.

In the next theorem, we summarize consequences from \cite{kohlmann2002global} in our setting to obtain a minimizer of $\KTJ$ in \eqref{eq:defcostfctLQ2} and a representation of the minimal costs.

\begin{theo}\label{thm:solnbykohlmanntang}
	Assume that there exists $\varepsilon \in (0,\infty)$ such that 
	$\lambda+\kappa\geq \varepsilon$ $dP\times ds|_{[0,T]}$-a.e.\ or $\sigma^2+2\sigma\eta\rcor+\eta^2\geq\varepsilon$ $dP\times ds|_{[0,T]}$-a.e.
	We then have:
	
	(i)
	There exists a unique solution $(K,L)$ of BSDE~\eqref{eq:BSDEKL}. 
	If $\sigma^2+2\sigma\eta\rcor+\eta^2\geq\varepsilon$ $dP\times ds|_{[0,T]}$-a.e., there exists $c\in (0,\infty)$ such that $P(K_s\geq c \text{ for all } s \in [0,T])=1$.
	
	(ii) 
	There exists a unique solution $(\psi,\phi)$ of BSDE~\eqref{eq:BSDEpsiphi}. 
	
	(iii)
	Let $x,d\in\R$, and let $\ssfihat^*$ be the solution of SDE~\eqref{eq:optssfihat}. 
	Then, $\hat u^* = (\hat u^*_s)_{s\in[0,T]}$ defined by 
	\begin{equation}\label{eq:defoptimalcontroluhat}
		\hat u^*_s = \theta_s \ssfihat_s^* - \theta_s^0, \quad s\in[0,T],
	\end{equation}
	is the unique optimal control in $\cL_0^2$ for $\KTJ$,
	and $\ssfihat^*$ is the corresponding state process (i.e., $\ssfihat^*=\ssfihat^{\hat u^*}$).
	
	(iv)
	Let $x,d \in \R$. The costs associated to the optimal control \eqref{eq:defoptimalcontroluhat} are given by 
	\begin{equation*}
	\begin{split}
		\inf_{\hat u \in \cL_0^2}\KTJ_0\left( \frac{d}{\sqrt{\gamma_0}} - \sqrt{\gamma_0} x,\hat u\right)
		&=\KTJ_0\left( \frac{d}{\sqrt{\gamma_0}} - \sqrt{\gamma_0} x,\hat u^*\right)\\
		&=K_0 \left( \frac{d}{\sqrt{\gamma_0}} - \sqrt{\gamma_0} x \right)^2 - 2 \psi_0 \left( \frac{d}{\sqrt{\gamma_0}} - \sqrt{\gamma_0} x \right) + C_0,
		\end{split}
	\end{equation*}
	where
	\begin{equation}\label{eq:optcostslongterm}
		\begin{split}
			C_0 & = \frac12 E_0\left[ \gamma_T \xi^2 \right] 
			+ E_0\left[ \int_0^T K_s \frac{\lambda_s^2}{\left(\lambda_s+\kappa_s\right)^2} \gamma_s \zeta_s^2 (\sigma_s^2 + 2\sigma_s\eta_s\rcor_s + \eta_s^2) ds \right] \\
			& \quad + E_0\left[ \int_0^T \frac{\lambda_s\kappa_s}{\lambda_s+\kappa_s}  \gamma_s \zeta_s^2 ds \right] 
			- E_0\left[ \int_0^T (\theta_s^0)^2 (\lambda_s + \kappa_s + (\sigma_s^2 + 2 \sigma_s \eta_s \rcor_s + \eta_s^2 ) K_s ) ds \right] \\
			& \quad + E_0\left[ \int_0^T 2 \kl \sqrt{\gamma_s} \zeta_s 
			\psi_s \bigg( \rho_s + \mu_s - \frac{\sigma_s^2 + \sigma_s \eta_s \rcor_s}{2} \bigg) ds \right] \\
			& \quad + E_0\left[ \int_0^T 2 \kl \sqrt{\gamma_s} \zeta_s \left( \phi_s^1 (\sigma_s + \eta_s \rcor_s) 
			+ \phi_s^2 \eta_s \sqrt{1-\rcor_s^2} \right) ds \right] .
		\end{split}
	\end{equation}
\end{theo}

\begin{proof}
	Observe that the problem in \Cref{sec:LQproblemwocrossterms} fits the problem considered in \cite[Section~5]{kohlmann2002global}. In particular, note that the coefficients in SDE~\eqref{eq:controlledprocdynhat} for $\ssfihat^{\hat u}$ and in the cost functional $\KTJ$ (see \eqref{eq:defcostfctLQ2}) are bounded, and that the inhomogeneities are in $\cL_0^2$. Moreover, we have that $\frac12$, $\frac{\lambda \kappa}{\lambda+\kappa}$, and $\lambda + \kappa$ are nonnegative. 
	Furthermore, the filtration by assumption in this section is generated by the Brownian motion $(W^1,\ldots,W^m)^\top$. 
	
	(i)
	If $\lambda+\kappa\geq \varepsilon$, this is an immediate consequence of \cite[Theorem~2.1]{kohlmann2002global}. In the case $\sigma^2+2\sigma\eta\rcor+\eta^2\geq\varepsilon$, this is an application of \cite[Theorem~2.2]{kohlmann2002global}.
	
	(ii)
	This is due to \cite[Theorem~5.1]{kohlmann2002global}.
	
	(iii)
	The first part of \cite[Theorem~5.2]{kohlmann2002global} yields the existence of a unique optimal control $\hat u^*$, 
	which is given in feedback form by the formula $\hat u^*=\theta \ssfihat^{\hat u^*} - \theta^0$. 
	We obtain~\eqref{eq:optssfihat} by plugging this into~\eqref{eq:controlledprocdynhat}.
	
	(iv)
	The second part of \cite[Theorem~5.2]{kohlmann2002global} provides us with the optimal costs.
\end{proof}

By an application of \Cref{cor:uoptimaliffuhatoptimal} and \Cref{cor:uoptimaliffXoptimal}, we obtain a solution to the trade execution problem of \Cref{sec:problemformulation}.

\begin{corollary}\label{cor:solnforJpm}
	Assume 
	that there exists $\varepsilon \in (0,\infty)$ such that 
	$\lambda+\kappa\geq \varepsilon$ $dP\times ds|_{[0,T]}$-a.e.\ or $\sigma^2+2\sigma\eta\rcor+\eta^2\geq\varepsilon$ $dP\times ds|_{[0,T]}$-a.e.
	Let $(K,L)$ be the unique solution of BSDE~\eqref{eq:BSDEKL}, $(\psi,\phi)$ the unique solution of BSDE~\eqref{eq:BSDEpsiphi}, and recall definitions \eqref{eq:theta} of $\theta$ and \eqref{eq:theta0} of $\theta^0$.
	Let $x,d \in \R$. Then, $X^*=(X^*_s)_{s\in[0-,T]}$ defined by 
	\begin{equation*}
	X^*_{0-}\! =x, \,\,\, X^*_T=\xi, \,\,\,\,
	X_s^* = \gamma_s^{-\frac12} \left(\! \bigg( \theta_s + \kl -1 \bigg) \ssfihat_s^* + \gamma_s^{\frac12} \zeta_s \kl \! - \theta_s^0 \right) \! , \, s\in [0,T),
	\end{equation*} 
	with $\ssfihat^*$ from \eqref{eq:optssfihat}, 
	is the unique (up to $dP\times ds|_{[0,T]}$-null sets) optimal execution strategy in $\cA_0^{pm}(x,d)$ for $\pmJ$. 
	The associated costs are given by 
	\begin{equation*}
	\inf_{X\in \cA_0^{pm}(x,d)} \!\! \pmJ_0(x,d,X) \!
	= 	\!	\pmJ_0(x,d,X^*) \!
	= \!
	K_0 \! \left(\! \frac{d}{\sqrt{\gamma_0}} \!-\! \sqrt{\gamma_0} x \! \right)^{\!\!2}\!\! - \! 2 \psi_0 \! \left(\! \frac{d}{\sqrt{\gamma_0}}\! - \!\sqrt{\gamma_0} x\! \right) \!+\! C_0 \!- \!\frac{d^2}{2\gamma_0}
	\end{equation*}
	with $C_0$ from~\eqref{eq:optcostslongterm}.
\end{corollary}

\begin{remark}
(i) 
Note that BSDE~\eqref{eq:BSDEKL} neither contains $\xi$ nor $\zeta$. In particular, the solution component $K$ and the process $\theta$ from~\eqref{eq:theta} do not depend on the choice of $\xi$ or $\zeta$
(although they depend on the choice of $\lambda$).
In contrast, BSDE~\eqref{eq:BSDEpsiphi} involves both $\xi$ and~$\zeta$. 
If $\xi=0$ and at least one of $\lambda$ and $\zeta$ is equivalent to $0$, we have that $(\psi,\phi)$ from~\eqref{eq:BSDEpsiphi}, $\theta^0$ from~\eqref{eq:theta0}, and $C_0$ from~\eqref{eq:optcostslongterm} vanish. 

\smallskip
(ii) Under the assumptions of \Cref{cor:solnforJpm} it holds that $K_0\le \frac{1}{2}$. This is a direct consequence of \Cref{cor:solnforJpm} and (i) above. Indeed, choose $\xi=0$ and $\zeta=0$ (by (i) this choice does not affect $K$). Then \Cref{cor:solnforJpm} and (i) show that $\pmJ_0(1,0,X^*)=K_0 \gamma_0$ for the optimal strategy $X^*$ from \Cref{cor:solnforJpm}. The suboptimal finite variation execution strategy $X_{0-}=1$, $X_s=0$, $s\in[0,T]$, in $\cA_0^{fv}(1,0)$ incurs costs $\pmJ_0(1,0,X)=\frac{\gamma_0}{2}$ and hence $K_0\le \frac{1}{2}$.

\smallskip
(iii) Our present setting essentially\footnote{The word ``essentially'' relates to different integrability conditions and to the fact that in \cite{ackermann2020cadlag} the formulation is for a continuous local martingale and a general filtration instead of Brownian motion with Brownian filtration.} includes the one in \cite{ackermann2020cadlag},
where we have $\xi=0$, $\lambda=0$, and $\eta=0$ (and, therefore, the processes $\zeta$ and $\rcor$ are not needed, cf.\ \eqref{eq:resilience} and~\eqref{eq:defcostfct1}).
In this subsetting the finite variation control problem associated with 
\eqref{eq:deviationdyndR}--\eqref{eq:defcostfct1}
is extended in \cite{ackermann2020cadlag} to a problem 
where the control $X$ is a c\`adl\`ag semimartingale that acts as integrator in the extended\footnote{The word ``extended'' relates to the fact that
\eqref{eq:deviationdyndR} and~\eqref{eq:defcostfct1}
need to be extended with certain additional terms when allowing for general semimartingale strategies, see \cite{ackermann2020cadlag}.} state dynamics of the form~\eqref{eq:deviationdyndR} and target functional of the form~\eqref{eq:defcostfct1}.
In \cite{ackermann2020cadlag} the existence of an optimal semimartingale strategy as well as the form of the optimal strategy (when it exists) is characterized in terms of a certain process $\wt\beta$,
which is in turn defined via a solution $(Y,Z,M^\perp)$ to a certain quadratic BSDE (see~(3.2) in \cite{ackermann2020cadlag}).
It is worth noting that, in the subsetting with $\xi=0$, $\lambda=0$, and $\eta=0$, all formulas in this section greatly simplify and, in particular, BSDE~\eqref{eq:BSDEKL} above is equivalent\footnote{For the sake of fair comparison, we consider the subsetting in \cite{ackermann2020cadlag} where the filtration is generated by $(W^1,\ldots,W^m)^\top$ and the continuous local martingale $M$ is $W^1$.} to BSDE~(3.2) in \cite{ackermann2020cadlag}.
The relation is $Y=K$, $Z=L^1$, $dM^\perp_s=\sum_{j=2}^m L^j_s\,dW^j_s$.
Further, in that subsetting, our process $\theta$ from~\eqref{eq:theta} reduces to the above-mentioned process $\wt\beta$ (see~(3.5) in \cite{ackermann2020cadlag}),
while $(\psi,\phi)$ from~\eqref{eq:BSDEpsiphi}, $\theta^0$ from~\eqref{eq:theta0}, and $C_0$ from~\eqref{eq:optcostslongterm} vanish.

\smallskip
(iv) It is also instructive to compare \Cref{cor:solnforJpm} above,
where we obtain that the extended to $\cA_0^{pm}(x,d)$ control problem always admits a minimizer,
with Theorem~3.4 in \cite{ackermann2020cadlag},
where it turns out that an optimal semimartingale strategy can fail to exist.
See the discussion in the end of \Cref{exa:ExampleNonexistenceFromLQct} for a specific example.
\end{remark}

\textbf{On the continuity of optimal position paths:}
In the setting of \cite{obizhaeva2013optimal} optimal position paths $X^*$ exhibit jumps (so-called block trades) at times $0$ and $T$ but are continuous on the interior $(0,T)$ (see also \cref{exa:ExampleOW} below). An interesting question is whether the continuity on $(0,T)$ prevails in the generalized setting considered in this paper. 
This is not reasonable to expect when we have the risk term with a ``sufficiently irregular'' process $\zeta$.
And, indeed, we see that the continuity of $X^*$ on $(0,T)$ can fail in \Cref{ex:120323a1} below (this is discussed in~\Cref{rem:X_discont_bc_zeta}).
More interestingly, such a continuity can already fail even without the risk term (i.e.\ $\lambda=0$) and with terminal target $\xi=0$. 
Indeed, consider the setting with $\sigma=0$, $\lambda=0$, $\xi=0$
and non-diffusive resilience process $R$ given by $R_s=\rho s$ (with $\rho$ being a deterministic constant). Then it follows from \cite[Example~6.2]{ackermann2020cadlag} that continuity of the price impact process $\gamma$ is not sufficient for continuity of optimal position paths $X^*$ on $(0,T)$. It is shown that if the paths of $\gamma$ are absolutely continuous, then a jump of the weak derivative of $\gamma$ on $(0,T)$ already causes $X^*$ to jump on $(0,T)$.
Moreover, it is possible that the random terminal target position $\xi$ causes the optimal position path $X^*$ to jump in $(0,T)$ with all other input processes being continuous. We present an example for this phenomenon in \Cref{subsec:random_term_target_jump}.

A way to obtain sufficient conditions for the continuity of $X^*$ on $(0,T)$ consists of combining \cref{cor:solnforJpm} with path regularity results for BSDEs. Indeed, if the coefficient processes $\rho, \mu, \sigma, \eta, \rcor, \lambda, \zeta$ are continuous and if one can ensure that the solution components $L^1, L^2$ and $\phi^1, \phi^2$ (which correspond to the martingale representation part of the solution) of the BSDE \eqref{eq:BSDEKL} resp.\ \eqref{eq:BSDEpsiphi} have continuous sample paths, then \cref{cor:solnforJpm} ensures that $X^*$ also has continuous sample paths on $(0,T)$. Results that guarantee continuity of BSDE solutions in a Markovian framework, including the quadratic case, can for example be found in \cite{imkeller2010path}.

\section{Examples}\label{sec:examples}
In this section we apply the results from the preceding sections in specific  case studies.

\subsection{The Obizhaeva-Wang model with random targets}\label{exa:ExampleOW}
The models developed by Obizhaeva and Wang \cite{obizhaeva2013optimal} can be considered as special cases of the model set up in \Cref{sec:problemformulation}. Indeed, we obtain the problem of \cite[Section 6]{obizhaeva2013optimal} by setting $\mu\equiv 0$, $\sigma\equiv 0$, $\eta\equiv 0$, $\rcor \equiv 0$, $\lambda\equiv 0$ and choosing $\rho \in (0,\infty)$ and $\xi \in \R$ as deterministic constants.

\begin{ex}\label{ex:120323a1}
In this example we apply our results (in particular, \Cref{cor:solnforJpm}) and provide closed-form solutions (see~\eqref{eq:gen_OW_opt_strat} below) for optimal progressively measurable execution strategies in versions of these problems which allow for general random terminal targets $\xi$ and general running targets $\zeta$.

To this end let $x,d \in \R$. Suppose that $\mu\equiv 0$, $\sigma\equiv 0$, $\eta\equiv 0$, and $\rcor \equiv 0$. 
	Furthermore, assume that $\rho\in (0,\infty)$ and $\lambda\in [0,\infty)$ are deterministic constants.
	We take some $\xi$ and $\zeta$ as specified in \Cref{sec:fvstrat} (in particular, see~\eqref{eq:conditionxi}).
	Note that the conditions of \Cref{thm:solnbykohlmanntang} and \Cref{cor:solnforJpm} hold true, and that $\gamma_s=\gamma_0$ for all $s \in [0,T]$. 
	In the current setting, BSDE~\eqref{eq:BSDEKL} reads 
	\begin{equation}\label{eq:BSDEKLOW}
		\begin{split}
			dK_s & = \bigg( \frac{\rho^2}{\rho+\lambda} K_s^2 + \frac{2\lambda \rho}{\rho + \lambda} K_s - \frac{\lambda \rho}{\rho + \lambda}  \bigg) ds + \sum_{j=1}^m L_s^j dW_s^j, \quad s\in [0,T], 
			\quad K_T =\frac12.
		\end{split}
	\end{equation}
	By \Cref{thm:solnbykohlmanntang}, there exists a unique solution $(K,L)$. Since the driver and the terminal condition in \eqref{eq:BSDEKLOW} are deterministic, we obtain that $L\equiv 0$, and hence \eqref{eq:BSDEKLOW} is in fact a scalar Riccati ODE with constant coefficients. Such an equation can be solved explicitly, and in our situation we obtain in the case $\lambda>0$ that
	\begin{equation*}
K_s=\frac{1}{2}
\frac{\lambda\tanh\left( \frac{\sqrt{\lambda} \rho(T-s)}{\sqrt{\lambda+\rho}}\right)+\sqrt{\lambda(\rho+\lambda)} }{(\frac{\rho}{2}+\lambda)\tanh\left( \frac{\sqrt{\lambda} \rho(T-s)}{\sqrt{\lambda+\rho}}\right)+\sqrt{\lambda(\rho+\lambda)}}, \quad s \in [0,T],
\end{equation*}
and in the case $\lambda=0$ that
	\begin{equation}\label{eq:K_OW_lambda_0}
			K_s=\frac{1}{2+(T-s)\rho}, \quad s \in [0,T].
			\end{equation}
	The process $\theta$ from \eqref{eq:theta} here is given by 
	$\theta_s = \frac{\rho}{\lambda+\rho} K_s$, $s\in[0,T]$. 
	BSDE \eqref{eq:BSDEpsiphi} becomes 
	\begin{equation}\label{eq:BSDEpsiphiOW}
	d\psi_s = 
	\! \left( \! \frac{\rho\lambda}{\lambda+\rho}\! + \! \rho \theta_s \! \right) \! \psi_s ds 
	+ \frac{\rho \lambda}{\lambda + \rho} \sqrt{\gamma_0} \zeta_s (1-K_s) ds 
	+ \sum_{j=1}^m \! \phi_s^j dW_s^j, \,\,\, s\in [0,T], \,\,\,
	\psi_T  = \!-\frac12\sqrt{\gamma_0} \xi.
	\end{equation}
	Again, by \Cref{thm:solnbykohlmanntang}, there exists a unique solution $(\psi,\phi)$.
	The solution component $\psi$ is given by 
	\begin{equation*}
		\begin{split}
			\psi_s & = \Gamma_s^{-1} \sqrt{\gamma_0} \left( - \frac12 \Gamma_T E_s[\xi] - \frac{\rho \lambda}{\lambda +\rho} E_s\left[\int_s^T \Gamma_r (1-K_r) \zeta_r\,dr\right] \right), \quad s\in [0,T],
		\end{split}
	\end{equation*} 
	where 
	\begin{equation}\label{eq:GammaOW}
		\Gamma_s
		=\exp\left( -\rho \int_0^s  
		\left( \frac{\lambda}{\lambda+\rho} + \theta_r  \right)
		dr \right)
		= \exp\left( -\frac{\rho}{\lambda + \rho} \left( \lambda s + \rho \int_0^s K_r dr \right) \right)
		, \quad s \in [0,T]. 
	\end{equation}
	It holds for the process in \eqref{eq:theta0} that 
	$\theta_s^0 = \frac{\rho}{\lambda+\rho} \psi_s$, $s\in[0,T]$. 
	Further, SDE~\eqref{eq:optssfihat} reads 
	\begin{equation*}
	d\ssfihat_s^*  = \! -\rho \left(\! \frac{\lambda}{\lambda+\rho} + \theta_s \! \right) \ssfihat_s^* ds 
	+ \rho \left(\! \theta_s^0 - \! \sqrt{\gamma_0} \zeta_s \frac{\lambda}{\lambda+\rho} \right)\! ds, 
	\,\,\, s \in [0,T], \,\,\,
	\ssfihat_0^* = \frac{d}{\sqrt{\gamma_0}} - \sqrt{\gamma_0} x, 
	\end{equation*}
	and has solution 
	\begin{equation*}
		\ssfihat_s^* = \Gamma_s \left( \frac{d}{\sqrt{\gamma_0}} - \sqrt{\gamma_0} x 
		+ \rho \int_0^s \Gamma_r^{-1} \left( \theta_r^0 - \sqrt{\gamma_0} \zeta_r \frac{\lambda}{\lambda+\rho} \right) dr \right), \quad s\in[0,T], 
	\end{equation*}
	with $\Gamma$ from~\eqref{eq:GammaOW}.
	It then follows from~\Cref{cor:solnforJpm} that 
	$X^*=(X^*_s)_{s\in[0-,T]}$ defined by $X_{0-}^*=x$, $X_T^*=\xi$, and, 
	for $s\in [0,T)$,  
	\begin{equation}\label{eq:gen_OW_opt_strat}
	\begin{split}
	&	X_s^* 
		= \gamma_0^{-\frac12} \left( \left( \theta_s - \frac{\rho}{\lambda + \rho} \right) \ssfihat_s^* - \theta_s^0 \right) + \zeta_s \frac{\lambda}{\lambda + \rho} \\
		& = \frac{\rho}{\lambda + \rho} (1-K_s) \Gamma_s\! \left( \! x-\! \frac{d}{\gamma_0} \!
		+ \! \frac{\rho}{\lambda + \rho} \int_0^s \! \Gamma_r^{-1} \! \left( \! \lambda \zeta_r \! - \! \frac{\rho}{\sqrt{\gamma_0}} \psi_r \! \right) dr \! \right) \!
		+ \! \frac{\rho}{\lambda + \rho} \left( \frac{\lambda}{\rho} \zeta_s \! -\! \frac{1}{\sqrt{\gamma_0}} \psi_s\! \right)
	\end{split}
	\end{equation}
	is the (up to $dP\!\times\! ds|_{[0,T]}$-null sets unique) execution strategy in $\cA_0^{pm}(x,d)$ that minimizes~$\pmJ$.
\end{ex}

\begin{remark}\label{rem:X_discont_bc_zeta}
From \Cref{ex:120323a1} we see that discontinuities of the target process $\zeta$ can cause jumps of the optimal position path $X^*$ in $(0,T)$.
Indeed, as $\theta$, $\theta^0$ and $\ssfihat^*$ are continuous, it follows from~\eqref{eq:gen_OW_opt_strat} that, in the case $\lambda>0$, paths of the optimal strategy~$X^*$ inherit discontinuities from $\zeta$ on $(0,T)$
(in particular, $X^*$ jumps on $(0,T)$ whenever~$\zeta$ does).
\end{remark}

In the next example we study the case $\lambda\equiv 0$ in more detail.

\begin{ex} 
In the setting of the previous example suppose that $\lambda \equiv 0$. If the terminal target $\xi\in \R$ is a deterministic constant, then it follows from \cite[Proposition 3]{obizhaeva2013optimal}  that the optimal finite variation execution strategy is given by
\begin{equation}\label{eq:060422a1}
			X_s^* = 
			\left( x-\xi -\frac{d}{\gamma_0} \right) \frac{1+(T-s)\rho}{2+T\rho} 
			+ \xi, \quad s\in[0,T).
\end{equation} 
So the optimal strategy consists of potential block trades (jumps of $X^*$) at times $0$ and $T$ and a continuous linear trading program on $[0,T)$. In the following we analyze how this structure changes as we allow for a random terminal target $\xi$.

First recall that the solution of BSDE~\eqref{eq:BSDEKLOW} is given in this case by \eqref{eq:K_OW_lambda_0}. 
It follows that $\Gamma$ from \eqref{eq:GammaOW} simplifies to $\Gamma_s = \frac{2+(T-s)\rho}{2+T\rho}$, $s \in [0,T]$.
		For the solution component $\psi$ of BSDE~\eqref{eq:BSDEpsiphiOW}, we thus obtain 
		\begin{equation*}
			\psi_s = - \frac{\sqrt{\gamma_0}}{2+(T-s)\rho} E_s[\xi] , \quad s\in [0,T].
		\end{equation*}
		The optimal strategy from \eqref{eq:gen_OW_opt_strat} on $[0,T)$ becomes, 
		for $s\in[0,T)$,  
		\begin{equation}\label{eq:optstratOWxi}
		\begin{split}
			X_s^* & = (1-K_s) \Gamma_s \left( x-\frac{d}{\gamma_0} - \rho\int_0^s \Gamma_r^{-1} \frac{1}{\sqrt{\gamma_0}} \psi_r dr \right) 
			- \frac{1}{\sqrt{\gamma_0}} \psi_s \\
			& = \left( x-\frac{d}{\gamma_0} \right) \frac{1+(T-s)\rho}{2+T\rho} 
			+ \rho (1+(T-s)\rho) \int_0^s \! \frac{E_r[\xi]}{(2+(T-r)\rho)^2} dr 
			 + \frac{E_s[\xi]}{2+(T-s)\rho} .
		\end{split}
		\end{equation}
		Integration by parts implies that (note that $(E_r[\xi])_{r \in [0,T]}$ is a continuous martingale)
		\begin{equation*}
			\begin{split}
				& \int_0^s \frac{E_r[\xi]}{(2+(T-r)\rho)^2} dr 
				= \int_0^s E_r[\xi] d\frac{1}{\rho(2+(T-r)\rho)} \\
				& = \frac{E_s[\xi]}{\rho (2+(T-s)\rho)} - \frac{E_0[\xi]}{\rho (2+T\rho)} 
				- \int_0^s \frac{1}{\rho(2+(T-r)\rho)} dE_r[\xi], 
				\quad s \in [0,T).
			\end{split}
		\end{equation*}
Substituting this into~\eqref{eq:optstratOWxi} yields, for $s\in[0,T)$,
\begin{align*}
X_s^*
&=
\left( x-E_0[\xi] -\frac{d}{\gamma_0} \right) \frac{1+(T-s)\rho}{2+T\rho}
+ E_s[\xi] - \int_0^s \frac{1+(T-s)\rho}{2+(T-r)\rho} dE_r[\xi]
\\
&=
\left( x-E_0[\xi] -\frac{d}{\gamma_0} \right) \frac{1+(T-s)\rho}{2+T\rho}
+ E_0[\xi] + \int_0^s \left(1-\frac{1+(T-s)\rho}{2+(T-r)\rho}\right) dE_r[\xi].
\end{align*}
We, finally, obtain the alternative representation 
		\begin{equation*}
			X_s^* = 
			\left( x-E_0[\xi] -\frac{d}{\gamma_0} \right) \frac{1+(T-s)\rho}{2+T\rho} 
			+ E_0[\xi] + \int_0^s \frac{1+(s-r)\rho}{2+(T-r)\rho} dE_r[\xi], \quad s \in [0,T), 
		\end{equation*}
		for \eqref{eq:optstratOWxi}. We see that this optimal strategy $X^*\in \cA_0^{pm}(x,d)$ consists of two additive parts: The first part exactly corresponds to the optimal deterministic strategy in~\eqref{eq:060422a1} where the deterministic terminal target is replaced by the expected terminal target $E_0[\xi]$. The second part represents fluctuations around this deterministic strategy which incorporate updates about the random terminal target $\xi$. Note that this stochastic integral vanishes in expectation,
although this is \emph{not} a martingale (indeed, the time $s$ is not only the upper bound of integration but also appears in the integrand).
\end{ex}

\subsection{A discontinuous optimal position path for continuous inputs}
\label{subsec:random_term_target_jump}

We now show that the optimal strategy can have jumps inside $(0,T)$ even if all input processes, including $\zeta$, are continuous. 
To this end, let $x,d\in\R$. Take $\lambda \equiv 0$, $\zeta \equiv 0$, $\eta \equiv 0$, $\rcor\equiv 0$, and $\mu \equiv 0$, and assume that $\sigma \in (0,\infty)$ and $\rho \in (\frac{\sigma^2}{2},\infty)$ are deterministic constants. 
Moreover, we will later consider an appropriate random terminal target $\xi$, satisfying the assumptions of \Cref{sec:fvstrat}, to produce a jump of the optimal strategy. 

Note that the conditions of \Cref{thm:solnbykohlmanntang} and \Cref{cor:solnforJpm} hold true.
In particular, there exists a unique solution $(K,L)$ of BSDE~\eqref{eq:BSDEKL}, and it is given by
(compare also with \cite[Section 5.2]{ackermann2020cadlag})
$L\equiv 0$ and 
\begin{equation*}
	K_s = \frac{\rho-\frac{\sigma^2}{2}}{\sigma^2} \mathcal{W}\left( \frac{\rho-\frac{\sigma^2}{2}}{\sigma^2} e^{c_0-\frac{\rho^2}{\sigma^2}s} \right)^{-1}, \quad s \in [0,T],
\end{equation*}
where $\mathcal{W}$ denotes the Lambert $W$ function and $c_0=\ln(2)+\frac{1}{\sigma^2} (2\rho-\sigma^2+\rho^2 T)$. 
The process $\theta$ from \eqref{eq:theta} becomes 
\begin{equation*}
	\theta_s = \frac{\rho K_s}{\rho-\frac{\sigma^2}{2} + \sigma^2 K_s}, \quad s\in[0,T],
\end{equation*}
and both $\theta$ and $K$ are deterministic, increasing, continuous, $(0,1/2]$-valued functions. 
 
For some $t_0\in(0,T)$, let 
\begin{equation*}
	\xi = -2\gamma_0^{-\frac12}  \left( \sigma \int_{t_0}^T \Gamma_s \theta_s ds + \int_{t_0}^T \Gamma_s dW_s^1 \right)\, \exp\left( - \frac{\sigma}{2} W^1_T + \frac38 \sigma^2 T + \left( \rho-\frac{\sigma^2}{2} \right)  \int_0^T \theta_s ds \right),
\end{equation*}
where 
$\Gamma_t = \exp( - \frac{\sigma^2}{8} t - ( \rho-\frac{\sigma^2}{2} ) \int_0^t \theta_s ds )$, $t\in[0,T]$. 
Note that $\xi$ is $\cF_T$-measurable and that $E[\gamma_T \xi^2]<\infty$. 
The terminal target $\xi$ here is defined in such a way that the unique solution $(\psi,\phi)$ of BSDE~\eqref{eq:BSDEpsiphi}  (cf.\ \Cref{thm:solnbykohlmanntang}) is given by 
$\phi^1=1_{[t_0,T]}$, $\phi^j\equiv 0$, $j\in\{2,\ldots,m\}$, and 
\begin{equation*}
	\psi_t = \begin{cases}
		0, & 0\leq t < t_0,\\
		\Gamma_t^{-1} \left( \sigma \int_{t_0}^t \Gamma_s \theta_s ds + \int_{t_0}^t \Gamma_s dW_s^1 \right), & t_0\leq t \leq T.
	\end{cases}
\end{equation*}
It follows for the process in~\eqref{eq:theta0} that 
\begin{equation*}
	\theta_t^0 = \begin{cases}
		0, & 0\leq t < t_0,\\
		\frac{\left(\rho-\frac{\sigma^2}{2}\right)\psi_t + \sigma}{\rho - \frac{\sigma^2}{2} + \sigma^2 K_t }, & t_0\leq t \leq T .
	\end{cases}
\end{equation*}
We thus have that 
\begin{equation*}
	\Delta \theta_{t_0}^0 
	= \frac{\sigma}{\rho-\frac{\sigma^2}{2}+\sigma^2 K_{t_0}} 
	> 0 .
\end{equation*}
From \Cref{cor:solnforJpm} we obtain existence of a unique optimal strategy $X^*$ and that  
$X_s^* = \gamma_s^{-\frac12} ( (\theta_s - 1) \ssfihat_s^* - \theta_s^0 )$, $s\in(0,T)$. 
Since $\gamma$, $\theta$, and $\ssfihat_s^*$ (see~\eqref{eq:optssfihat}) are continuous and $\Delta \theta_{t_0}^0 >0$, it holds that 
$\Delta X_{t_0}^* = -\gamma_{t_0}^{-\frac12} \Delta \theta_{t_0}^0 < 0$. 
Hence, the optimal strategy has a jump at $t_0\in (0,T)$.

\subsection{An example where $\fvJ$ does not admit a minimizer}\label{exa:ExampleNonexistenceFromLQct}

	Let $x,d\in\R$ with $x\neq \frac{d}{\gamma_0}$. 
	Suppose that $\sigma\equiv 0$, $\eta \equiv 0$, $\lambda\equiv 0$, $\rcor \equiv 0$, $\zeta\equiv 0$, $\xi=0$. 
	Choose $\mu$ to be a bounded deterministic c\`adl\`ag function such that there exists $\delta \in (0,T)$ with $\mu$ having infinite variation on $[0,T-\delta]$, and take $\rho \in \R\setminus\{0\}$ such that there exists $\varepsilon>0$ with $2\rho+\mu\geq \varepsilon$. 
	Note that this corresponds to the setting in \cite[Example~6.4]{ackermann2020cadlag}. Moreover, observe that the conditions of \Cref{cor:solnforJpm} are satisfied. 
	In the current setting, BSDE~\eqref{eq:BSDEKL} becomes 
	\begin{equation*}
		\begin{split}
			dK_s & = \left(-\mu_s K_s + \frac{2(\rho+\mu_s)^2 K_s^2}{2\rho+\mu_s} \right) ds + \sum_{j=1}^m L_s^{j} dW_s^j, \quad s \in [0,T],\quad
			K_T=\frac12.
		\end{split}
	\end{equation*}
	Its solution is given by $(K,0)$, where (see also $Y$ in \cite[Section~6]{ackermann2020cadlag})
	\begin{equation*}
		K_s=e^{\int_s^T\mu_r dr} \left( \int_s^T \frac{2(\rho+\mu_r)^2}{2\rho+\mu_r} e^{\int_r^T \mu_l dl} dr + 2 \right)^{-1}, \quad s\in[0,T],
	\end{equation*}
	is a deterministic continuous function of finite variation. 
	We have that 
	\begin{equation*}
		\theta_s=\frac{2(\rho+\mu_s)}{2\rho+\mu_s} K_s, \quad s\in [0,T],
	\end{equation*}
	which is the same as $\wt \beta$ in \cite[Example~6.4]{ackermann2020cadlag}. 
	The solution of BSDE~\eqref{eq:BSDEpsiphi} is given by $(\psi,\phi)=(0,0)$, and it holds $\theta^0\equiv 0$. 
	Furthermore, \eqref{eq:optssfihat} reads 
	\begin{equation*}
		d\ssfihat^*_s = \left( \frac{\mu_s}{2} - (\rho+\mu_s) \theta_s \right) \ssfihat_s^* ds, \quad s \in [0,T], \quad \ssfihat_0^*=\frac{d}{\sqrt{\gamma_0}} - \sqrt{\gamma_0} x, 
	\end{equation*}
	and is solved by the continuous deterministic finite-variation function
	\begin{equation*}
		\ssfihat_s^* = \left( \frac{d}{\sqrt{\gamma_0}} - \sqrt{\gamma_0} x \right) \exp\left( \int_0^s 
		\left( \frac{\mu_r}{2} - (\rho + \mu_r) \theta_r \right)
		dr \right) ,\quad s \in [0,T],
	\end{equation*}
	which is nonvanishing due to our assumption $x\ne\frac d{\gamma_0}$.\footnote{At this point it is easy to explain why we exclude the case $x=\frac d{\gamma_0}$ in this example. In the case $x=\frac d{\gamma_0}$ we get that $\ssfihat^*\equiv0$ and then the optimal strategy is to close the position immediately, i.e., $X^*_{0-}=x$, $X^*_s=0$, $s\in[0,T]$, which is always a finite-variation strategy.}
	By \Cref{cor:solnforJpm}, there exists a (up to $dP\times ds|_{[0,T]}$-null sets) unique minimizer $X^*=(X^*_s)_{s\in[0-,T]}$ of $\pmJ$ in $\cA_0^{pm}(x,d)$, namely 
	\begin{equation*}
		\begin{split}
			&X^*_{0-}=x, \quad X^*_T=0, \quad
			X_s^* = \gamma_s^{-\frac12} \left(\theta_s -1 \right) \ssfihat_s^*, \quad s\in [0,T).
		\end{split}
	\end{equation*} 
	Assume by contradiction that there exists a minimizer $X^0=(X^0_s)_{s\in[0-,T]}$ of $\fvJ$ in $\cA_0^{fv}(x,d)$. 
	We know from \Cref{cor:equivofinfcostfct} that $X^0$ is then also a minimizer of $\pmJ$ in $\cA_0^{pm}(x,d)$. 
	It follows that $X^0=X^*$ $dP\times ds|_{[0,T]}$-a.e. 
	Since $\ssfihat^*$ is nowhere $0$, we obtain that 
	\begin{equation}\label{eq:030422a1}
		1+\frac{\gamma^{\frac12} X^0}{\ssfihat^*} = \theta \quad  dP\times ds|_{[0,T]}\text{-a.e.}
	\end{equation}
Observe that the left-hand side is a process of finite variation. 
	On the other hand, our assumption on $\mu$ easily yields that $\theta$ has infinite variation.
	This contradiction proves that in the setting of this example, $\fvJ$ does not admit a minimizer in $\cA_0^{fv}(x,d)$.
	
	We can say even more:
	In this example there does not exist a semimartingale optimal strategy.\footnote{Under a \emph{semimartingale strategy} we formally understand a semimartingale that is an element of $\cA_0^{pm}(x,d)$.}
	Indeed, if we had a semimartingale $X^0$ as a minimizer, we would still get~\eqref{eq:030422a1} (with a semimartingale $X^0$).
	The left-hand side would then be a semimartingale.
	On the other hand, it is shown in \cite[Example~6.4]{ackermann2020cadlag} that there does not exist a semimartingale $\beta$ such that $\beta=\theta$  $dP\times ds|_{[0,T]}$-a.e.
	Thus, the cost functional does not have a minimizer in the set of semimartingales, but we are now able to find a minimizer in the set of progressively measurable execution strategies.

\subsection{An example with a diffusive resilience}\label{sec:generalres}

As already mentioned in the introduction, the literature on optimal trade execution in Obizhaeva-Wang type models 
typically assumes that $R$ is an increasing process.
In \cite{ackermann2020cadlag} and \cite{ackermann2021negativeresilience} $R$ is allowed to have finite variation.
Now we consider an example with a truly diffusive $R$.

	Let $x,d\in\R$ with $x\neq\frac{d}{\gamma_0}$. Let $\xi=0$, $\lambda \equiv 0$, $\zeta \equiv 0$, and $\mu\equiv 0$. Suppose that $\rcor \in [-1,1]$ and $\eta,\rho,\sigma \in \R$ are deterministic constants such that 
	$\kappa = \frac12 (2\rho-\sigma^2-\eta^2-2\sigma \eta \rcor) >0$ 
	and $\sigma^2+\eta^2+2\sigma\eta \rcor>0$ (in particular, we thus need $\rho>0$). 
	Note that the assumptions of \Cref{cor:solnforJpm} are satisfied. 
	We moreover remark that the subsetting where $\eta\equiv 0$ corresponds to the setting in \cite[Section~5.2]{ackermann2020cadlag}. That means, the difference to \cite[Section~5.2]{ackermann2020cadlag} is that we now consider a more general resilience.  
	The Riccati-BSDE~\eqref{eq:BSDEKL} becomes
	\begin{equation*}
		dK_s = \frac{(\rho K_s + (\sigma+\eta \rcor)L_s^1 + \eta \sqrt{1-\rcor^2}L_s^2 )^2}{(\sigma^2+\eta^2+2\sigma\eta\rcor)K_s + \kappa} ds 
		-\sigma L_s^1 ds 
		+ \sum_{j=1}^m \! L_s^j dW_s^j, \,\,\, s\in[0,T], \,\,\,
		K_T  = \frac12.
	\end{equation*}
	This has solution $(K,L)=(K,0)$ with 
	\begin{equation*}
		K_s = \frac{\kappa}{\sigma^2+\eta^2+2\sigma\eta\rcor} 
		\mathcal{W}\left( \frac{\kappa}{\sigma^2+\eta^2+2\sigma\eta\rcor} 
		\exp\left( c - \frac{\rho^2 s}{\sigma^2+\eta^2+2\sigma\eta\rcor} \right)  \right)^{-1}, \quad s \in [0,T],
	\end{equation*}
	and 
	$c = \ln(2) + \frac{2\kappa + \rho^2 T}{\sigma^2+\eta^2+2\sigma\eta\rcor}$ 
(compare also with \cite[Section 5.2]{ackermann2020cadlag}).
	We further have that 
	$\theta_s = \frac{\rho K_s}{ (\sigma^2+\eta^2+2\sigma\eta\rcor)K_s + \kappa}$, $s\in [0,T]$. 
	Observe that $(\psi,\phi)=(0,0)$ is the solution of~\eqref{eq:BSDEpsiphi} in the present setting and that $\theta^0\equiv 0$ in~\eqref{eq:theta0}. 
	Moreover, we have that SDE~\eqref{eq:optssfihat} reads 
	\begin{equation*}
		d\ssfihat_s^* = \!\left(\! - \frac{\sigma^2}{8} \! - \!\left(\! \rho - \frac{\sigma^2+\sigma\eta\rcor}{2} \right) \theta_s \! \right) \! \ssfihat_s^* ds 
		+ \left( \frac{\sigma}{2}\! -\! (\sigma+\eta\rcor) \theta_s \right) \! \ssfihat_s^* dW_s^1 
		- \eta \sqrt{1-\rcor^2} \theta_s \ssfihat_s^* dW_s^2 
	\end{equation*}
	for $s\in[0,T]$, 
	with start in $\ssfihat_0^* = \frac{d}{\sqrt{\gamma_0} } - \sqrt{\gamma_0} x$;
	hence, 
	\begin{equation*}
		\begin{split}
			\ssfihat_s^* & = \left( \frac{d}{\sqrt{\gamma_0} } - \sqrt{\gamma_0} x \right) 
			\exp\!\left( - \frac{\sigma^2 s}{4} - (\rho-\sigma^2-\sigma\eta\rcor) \int_0^s \theta_r dr 
			- \frac{\sigma^2 + \eta^2 + 2 \sigma \eta \rcor }{2} \int_0^s \theta_r^2 dr \right) \\
			& \quad \cdot \exp\left( \frac{\sigma}{2} W_s^1 - (\sigma + \eta \rcor) \int_0^s \theta_r dW_r^1 - \eta \sqrt{1-\rcor^2} \int_0^s \theta_r dW_r^2 \right), \quad s\in[0,T].
		\end{split}
	\end{equation*}
	It follows from \Cref{cor:solnforJpm} that for $s \in [0,T)$ the optimal execution strategy is given by
	\begin{equation*}
		\begin{split}
			X_s^* & = \left( x - \frac{d}{\gamma_0} \right) (1-\theta_s) 
			\exp\!\left( - (\rho-\sigma^2-\sigma\eta\rcor) \int_0^s \theta_r dr 
			- \frac{\sigma^2 + \eta^2 + 2 \sigma \eta \rcor }{2} \int_0^s \theta_r^2 dr \right) \\
			& \quad \cdot \exp\left( - (\sigma + \eta \rcor) \int_0^s \theta_r dW_r^1 - \eta \sqrt{1-\rcor^2} \int_0^s \theta_r dW_r^2 \right) .
		\end{split}
	\end{equation*}
	We can show that $K$ and $\theta$ both are continuous, deterministic, increasing, $(0,1/2]$-valued functions of finite variation. 
	Since $\theta<1$, the optimal strategy on $[0,T)$ always has the same sign as $x-\frac{d}{\gamma_0}$.  
	Moreover, 
	the optimal strategy is stochastic and has infinite variation, as in \cite[Section~5.2]{ackermann2020cadlag}. 
	In contrast to \cite[Section~5.2]{ackermann2020cadlag}, where the price impact always has infinite variation, 
	we can here set $\sigma\equiv 0$ for a choice of $\eta^2\in(0,2\rho)$. In this case, the price impact $\gamma\equiv \gamma_0$ is a deterministic constant, yet the optimal strategy has infinite variation 
	(due to the infinite variation in the resilience $R$).

	Observe furthermore that by making use of $\eta$ and $\rcor$, we can choose the parameters in the current setting in such a way that $\kappa>0$ and $\sigma^2+\eta^2+2\sigma\eta \rcor>0$ are satisfied, but condition~(3.1) in \cite{ackermann2020cadlag}, i.e., $2\rho-\sigma^2>0$, is violated.
	
	With regard to \Cref{exa:ExampleNonexistenceFromLQct} we remark that 
	in both sections there does not exist an optimal strategy in $\cA_0^{fv}(x,d)$, but opposed to \Cref{exa:ExampleNonexistenceFromLQct}, it holds in the current section that there exists a semimartingale optimal strategy.

\subsection{Cancellation of infinite variation}\label{sec:cancelinfvar}

	We now present an example where the infinite variation in the price impact process $\gamma$ is ``cancelled'' by the infinite variation in the resilience process $R$ and we obtain the optimal strategy $X^*$ of finite variation.
	
	Let $x,d\in\R$, 
	$\xi=0$, $\lambda \equiv 0$, $\zeta \equiv 0$, and $\mu\equiv 0$. Suppose that $\rcor=-1$ and $\rho>0$ are deterministic constants, and that $\eta$ and $\sigma$ are progressively measurable, $dP\times ds|_{[0,T]}$-a.e.\ bounded processes such that $\eta=\sigma$ $dP\times ds|_{[0,T]}$-a.e.
	It then holds $dP\times ds|_{[0,T]}$-a.e.\ that $\sigma^2 + \eta^2 + 2\sigma\eta\rcor = 0$ and $\kappa=\rho>0$. In particular, the assumptions of \Cref{cor:solnforJpm} are satisfied. 
	The BSDE 
	\begin{equation*}
		\begin{split}
			d K_s & = \rho K_s^2 ds - \sigma_s L_s^1 ds + \sum_{j=1}^m L_s^j dW_s^j, \quad s\in [0,T],
			\quad K_T =\frac12,
		\end{split}
	\end{equation*}
	which is BSDE~\eqref{eq:BSDEKL} in the present setting, 
	has the solution $(K,L)=(K,0)$ with 
	$K_s = \frac{1}{2+(T-s)\rho}$, $s\in[0,T]$ 
	(cf.\ \Cref{exa:ExampleOW}). 
	It holds that $\theta \equiv K$, that $(\psi,\phi)=(0,0)$ is the solution of~\eqref{eq:BSDEpsiphi}, and that $\theta^0 \equiv 0$. 
	It follows that~\eqref{eq:optssfihat} has the solution 
	\begin{equation*}
		\begin{split}
			\ssfihat_s^* & = 
			\left( \frac{d}{\sqrt{\gamma_0} } - \sqrt{\gamma_0} x \right) 
			\exp\left( - \frac{1}{4} \int_0^s \sigma_r^2 dr - \rho \int_0^s K_r dr 
			+ \frac{1}{2} \int_0^s \sigma_r dW_r^1 \right), \quad s\in[0,T] .
		\end{split}
	\end{equation*}
	For the optimal execution strategy from \Cref{cor:solnforJpm} we then compute that 
	\begin{equation*}
		X_s^* = \left( x - \frac{d}{\gamma_0} \right) 
		\frac{1+(T-s)\rho}{2+T\rho}, \quad s\in[0,T).
	\end{equation*}
The optimal strategy in the current setting with general
stochastic
$\sigma=\eta$ and negative correlation
$\rcor=-1$
is thus the same as in the Obizhaeva-Wang setting $\sigma=0=\eta$ (cf.\ \cite[Proposition~3]{obizhaeva2013optimal}; see also \cite[Section~4.2]{ackermann2020cadlag}). 
	In particular, the optimal strategy is deterministic and of finite variation, although the price impact $\gamma$ and the resilience $R$ are both stochastic and of infinite variation (at least if $\sigma=\eta$
is nonvanishing).

We finally remark
that this setting does not reduce to the Obizhaeva-Wang setting $\sigma=0=\eta$.
Indeed, 
while the optimal strategies for $\sigma=0=\eta$ and for general
stochastic 
$\sigma=\eta$
with correlation $\rcor=-1$ 
coincide, this is not true for the associated deviation processes. 
	In general, it holds that 
	\begin{equation*}
		D_s^{X^*} = -\gamma_0 \left( x - \frac{d}{\gamma_0} \right) \frac{1}{2+T\rho} \exp\left( \int_0^s \eta_r dW_r^1 - \frac12 \int_0^s \eta_r^2 dr \right), \quad s\in[0,T), 
	\end{equation*}
which for
a nonvanishing $\eta$ 
and $x\neq \frac{d}{\gamma_0}$ has infinite variation,
whereas in the Obizhaeva-Wang setting is constant (take $\eta= 0$).

\section{Proofs}\label{sec:proofs}

In this section, we provide the proofs for the results presented in \Cref{sec:problemformulation} and \Cref{sec:reductiontoLQ}. 
We furthermore state and prove some auxiliary results that are used in the proofs of the main results.

For reference in several proofs, note that the order book height, i.e., the inverse of the price impact, has dynamics 
\begin{equation}\label{eq:LOBheight}
	d\gamma_s^{-1} = \gamma_s^{-1} \left( -(\mu_s-\sigma_s^2)ds - \sigma_s dW^1_s \right), \quad s\in[0,T].
\end{equation}
We moreover observe that by It\^o's lemma it holds that 
\begin{equation}\label{eq:sqrtgammadyn}
	\begin{split}
		d\gamma_s^{\frac12} & = \gamma_s^{\frac12} \left( \frac12\mu_s - \frac{1}{8} \sigma_s^2 \right)ds + \frac12 \gamma_s^{\frac12} \sigma_s dW^1_s, \quad s \in [0,T],
	\end{split}
\end{equation}
\begin{equation}\label{eq:sqrtalphadyn}
	\begin{split}
		d\gamma_s^{-\frac12} & = \gamma_s^{-\frac12} \left( -\frac12\mu_s + \frac{3}{8} \sigma_s^2 \right)ds - \frac12 \gamma_s^{-\frac12} \sigma_s dW^1_s, \quad s \in [0,T].
	\end{split}
\end{equation}

\begin{proof}[Proof of \Cref{propo:costfunctionalpart}]
	Observe that integration by parts implies that for all $s \in [t,T]$
	\begin{equation*}
		\begin{split}
			d(\nu_s D_s) 
			& = \nu_s dD_s + D_s d\nu_s + d[\nu, D]_s \\
			& = -\nu_sD_s dR_s + \nu_s \gamma_s dX_s +  \nu_s D_s dR_s + \nu_s D_s d[R]_s + d[\nu, D]_s\\
			& = \nu_s \gamma_s dX_s + \nu_s D_s d[R]_s + d[\nu, D]_s.
		\end{split}
	\end{equation*}
	Since 
	$d[\nu,D]_s = \nu_s d[R,D]_s = - \nu_s D_s d[R]_s$, $s\in[t,T]$, 
	it follows that the process $\mD_s=\nu_sD_s$, $s\in[t,T]$, $\mD_{t-}=d$, satisfies
	\begin{equation}\label{eq:Dtildedyn}
		\begin{split}
			d\mD_s  = d(\nu_s D_s) 
			& = \nu_s \gamma_s dX_s, \quad s \in [t,T].
		\end{split}
	\end{equation}
	In particular, $\mD$ is of finite variation. 
	The facts that $\Delta D_s = \gamma_s \Delta X_s$, $s \in [t,T]$, and $d\mD_s = \nu_s \gamma_s dX_s$, $s \in [t,T]$, imply that 
	\begin{equation}\label{eq:costfunctionalpart001}
	\begin{split}
		\int_{[t,T]} \left( 2D_{s-} + \Delta X_s \gamma_s \right) dX_s 
		& = \int_{[t,T]} \left( 2D_{s-} + \Delta D_s \right) dX_s 
		= \int_{[t,T]} \left( 2D_{s-} + \Delta D_s \right) \gamma_s^{-1} \nu_s^{-1} d\mD_s \\
		& = \int_{[t,T]} \nu_s^{-2}\gamma_s^{-1} \left( 2\mD_{s-} + \Delta \mD_s \right) d\mD_s 
		= \int_{[t,T]} \aphi_s\,d(\mD^2_s),
	\end{split}
	\end{equation}
	where we denote $\aphi_s = \nu_s^{-2}\gamma_s^{-1}$, $s\in[t,T]$,
and, in the last equality, we use that $d(\mD^2_s)=(2\mD_{s-}+\Delta\mD_s)\,d\mD_s$, as $\mD$ has finite variation. 
Summing up, \eqref{eq:costfunctionalpart001} yields 
	\begin{equation*}
		\begin{split}
			\int_{[t,T]} D_{s-}dX_s + \frac12 \int_{[t,T]} \Delta X_s \gamma_s dX_s &=
			\frac12 \left(\mD_T^2 \aphi_T-\mD_{t-}^2 \aphi_t-\int_t^T \mD_s^2 d\aphi_s\right)
			\\
			& = \frac12 \left( \gamma_T^{-1} D_T^2 - \gamma_t^{-1} d^2 - \int_{t}^T D_s^2\nu_s^2 d\left(\nu_s^{-2} \gamma_s^{-1}\right) \right).
		\end{split}
	\end{equation*} 	
	
	In order to show \eqref{eq:deviationrewritten}, we first obtain from \eqref{eq:Dtildedyn} and integration by parts that 
	\begin{equation*}
		\begin{split}
			\nu_r D_r - d & = \nu_r\gamma_r X_r - \gamma_t x - \int_{[t,r]} X_s d(\nu_s\gamma_s) - \int_{[t,r]} d[\nu\gamma,X]_s , \quad r \in [t,T].
		\end{split}
	\end{equation*}
	This implies that 
	$D_r = \gamma_r X_r + \nu_r^{-1} ( d - \gamma_t x - \int_{t}^r X_s d(\nu_s\gamma_s) )$, $r\in[t,T]$.
\end{proof}

\begin{proof}[Proof of \Cref{propo:representationcostfct}]	
	We first consider the integrator $\nu^{-2}\gamma^{-1}$ on the right hand side of \eqref{eq:costfunctionalpart}. 
	It holds by integration by parts and 
	\eqref{eq:nuinvdyn} 
	that for all $s \in [t,T]$ 
	\begin{equation*}
		\begin{split}
			d(\nu_s^{-2}\gamma_s^{-1})
			& = \nu_s^{-1} d(\gamma_s^{-1}\nu_s^{-1}) + \gamma_s^{-1} \nu_s^{-1} d\nu_s^{-1}  + d[\nu^{-1},\gamma^{-1}\nu^{-1}]_s \\
			& = 2 \nu_s^{-1} \gamma_s^{-1} d\nu_s^{-1} + \nu_s^{-2}  d\gamma_s^{-1} + \nu_s^{-1}  d[\gamma^{-1},\nu^{-1}]_s + d[\nu^{-1},\gamma^{-1}\nu^{-1}]_s \\
			& = - 2 \nu_s^{-2} \gamma_s^{-1} dR_s  
			+ \nu_s^{-2} d\gamma_s^{-1} 
			- \nu_s^{-2} d[\gamma^{-1},R]_s  
			+ d[\nu^{-1},\gamma^{-1}\nu^{-1}]_s  .
		\end{split}		
	\end{equation*} 
	Note that for all $s \in [t,T]$ we have 
	\begin{equation*}
		\begin{split}
			d[\nu^{-1},\gamma^{-1}\nu^{-1}]_s 
			& = - \nu_s^{-1} d[R, \gamma^{-1}\nu^{-1}]_s 
			= - \nu_s^{-1} d\left[ R, \int_t^{\cdot} 	\gamma^{-1} d\nu^{-1} + \int_t^{\cdot} \nu^{-1} d\gamma^{-1} \right]_s \\
			& = - \nu_s^{-1} \gamma_s^{-1} d[R,\nu^{-1}]_s - \nu_s^{-2} d[R,\gamma^{-1}]_s 			
			= \nu_s^{-2} \gamma_s^{-1} d[R]_s 
			- \nu_s^{-2} 
			d[ R, \gamma^{-1}]_s.
		\end{split}
	\end{equation*}
	It hence follows for all $s \in [t,T]$ that 
	\begin{equation*}
		\begin{split}
			d(\nu_s^{-2}\gamma_s^{-1})
			& = - 2 \nu_s^{-2} \gamma_s^{-1} dR_s  
			+ \nu_s^{-2} d\gamma_s^{-1} 
			- 2 \nu_s^{-2} d[\gamma^{-1},R]_s  
			+ \nu_s^{-2} \gamma_s^{-1} d[R]_s .
		\end{split}
	\end{equation*}
	Plugged into \eqref{eq:costfunctionalpart} from \Cref{propo:costfunctionalpart}, we obtain that 
	\begin{equation}\label{eq:costfunctionalpartb}
		\begin{split}
			& \int_{[t,T]} D_{s-}dX_s + \frac12 \int_{[t,T]} \Delta X_s \gamma_s dX_s  \\
			& = \frac12 \left( \gamma_T^{-1} D_T^2 - \gamma_t^{-1} d^2 
			- \int_t^T D_s^2 \left( d\gamma_s^{-1} + \gamma_s^{-1}d[R]_s - 2\gamma_s^{-1} dR_s 
			- 2 d[\gamma^{-1},R]_s \right) \right) .
		\end{split}
	\end{equation}	
	We further have by \eqref{eq:resilience} and \eqref{eq:LOBheight} that for all $s \in [t,T]$
	\begin{equation}\label{eq:1207a2}
		\begin{split}
			& d\gamma_s^{-1} + \gamma_s^{-1} d[R]_s - 2 \gamma_s^{-1} dR_s - 2d[\gamma^{-1},R]_s \\
			& = - \gamma_s^{-1} (\mu_s - \sigma_s^2) ds - \gamma_s^{-1} \sigma_s dW^1_s + \gamma_s^{-1} \eta_s^2 ds - 2 \gamma_s^{-1} \rho_s ds 
			- 2 \gamma_s^{-1} \eta_s dW_s^R + 2 \gamma_s^{-1} \sigma_s \eta_s \rcor_s ds \\
			& = - \gamma_s^{-1} \left( 2\rho_s +\mu_s - \sigma_s^2 - \eta_s^2 - 2\sigma_s \eta_s \rcor_s \right) ds 
			- \gamma_s^{-1} \sigma_s dW^1_s - 2\gamma_s^{-1} \eta_s dW_s^R .
		\end{split}
	\end{equation}
	It follows from assumption (A1) and the boundedness of the input processes that 
	\begin{equation*}
		\begin{split}
			E\left[ \left\lvert \int_t^T D_s^2 \gamma_s^{-1} \left( 2\rho_s +\mu_s - \sigma_s^2 - \eta_s^2 - 2\sigma_s \eta_s \rcor_s \right) ds  \right\rvert \right] < \infty.
		\end{split}
	\end{equation*}
	The Burkholder-Davis-Gundy inequality together with assumption (A3) shows that it holds for some constant $c \in (0,\infty)$ that 
	\begin{equation*}
		\begin{split}
			E\left[ \sup_{r \in [t,T]} \left\lvert \int_t^r D_s^2 \gamma_s^{-1} \sigma_s dW^1_s \right\rvert \right] 
			& \leq c E\left[ \left( \int_t^T D_s^4 \gamma_s^{-2} \sigma_s^2 ds \right)^{\frac12} \right] 
			< \infty .
		\end{split}
	\end{equation*}
	We therefore have that 
	$E_t[ \int_t^T D_s^2 \gamma_s^{-1} \sigma_s dW^1_s ] = 0.$ 
	Similarly, assumption (A2) implies that 
	$E_t[ \int_t^T 2 D_s^2 \gamma_s^{-1} \eta_s dW_s^R ] = 0.$ 
	It thus follows from 
	\eqref{eq:costfunctionalpartb}, \eqref{eq:1207a2}, and \eqref{eq:defkappa} that 
	\begin{equation*}
		E_t\left[ \int_{[t,T]} D_{s-}dX_s + \frac12 \int_{[t,T]} \Delta X_s \gamma_s dX_s \right] 
		= \frac12 E_t\left[ \gamma_T^{-1} D_T^2 + \int_t^T D_s^2 \gamma_s^{-1} 2\kappa_s ds \right] 
		- \frac{d^2}{2\gamma_t} .
	\end{equation*}	
	By definition \eqref{eq:defcostfct1} of $\fvJ$ this proves \eqref{eq:costfctrewritten}. 
\end{proof}

The dynamics that we compute in the following lemma are used in the proofs of \Cref{lem:scaledhiddendevdyn} and \Cref{lem:getXfromu}.

\begin{lemma}\label{lem:dynalphabeta}
	Let $t \in [0,T]$, $x,d \in \R$. Assume that $X=(X_s)_{s\in[t-,T]}$ is a progressively measurable process such that $\int_t^T X_s^2 ds < \infty$ a.s. For $\alpha_s=\gamma_s^{-\frac12}\nu_s^{-1}$, $s \in [t,T]$, and $\beta_s=d-\gamma_t x -\int_t^s X_r d(\nu_r\gamma_r)$, $s \in [t,T]$, it then holds for all $s \in [t,T]$ that 
	\begin{equation}\label{eq:dynalphabeta}
		\begin{split}
			& d(\alpha_s\beta_s) \\
			& = 
			-\gamma_s^{\frac12} X_s 
			\Bigg( \big( \mu_s +\rho_s +\eta_s^2 + \sigma_s \eta_s \rcor_s \big) ds 
			+ \big( \sigma_s + \eta_s \rcor_s \big)dW_s^1 
			+ \eta_s \sqrt{1-\rcor_s^2}dW_s^2 \Bigg) \\
			& \quad + \alpha_s \beta_s  \Bigg(  
			\Big( - \rho_s - \frac12 \mu_s + \frac38 \sigma_s^2 + \frac12 \sigma_s \eta_s \rcor_s \Big) ds 
			+ \Big( -\eta_s \rcor_s - \frac12 \sigma_s \Big) dW_s^1
			- \eta_s \sqrt{1-\rcor_s^2} dW_s^2
			\Bigg) \\
			& \quad + \gamma_s^{\frac12} X_s \left( \frac32 \eta_s \sigma_s \rcor_s + \frac12 \sigma_s^2 + \eta_s^2 \right) ds .
		\end{split}
	\end{equation}
\end{lemma}

\begin{proof}
	Integration by parts implies that 
	\begin{equation}\label{eq:20080}
		d(\alpha_s\beta_s) = -\alpha_s X_s d(\nu_s\gamma_s) + \beta_s d(\gamma_s^{-\frac12} \nu_s^{-1}) - X_sd[\gamma^{-\frac12} \nu^{-1},\nu\gamma]_s , \quad s \in [t,T].
	\end{equation}
	Furthermore, it holds by integration by parts, \eqref{eq:defnu}, \eqref{eq:resilience} and \eqref{eq:priceimpact} that for all $s \in [t,T]$ 
	\begin{equation}\label{eq:2008a}
		\begin{split}
			d(\nu_s\gamma_s) 
			& = \nu_s d\gamma_s + \gamma_s \nu_s dR_s + \gamma_s \nu_s d[R]_s + \nu_s d[R,\gamma]_s \\
			& = \nu_s \gamma_s \mu_s ds + \nu_s \gamma_s \sigma_s dW^1_s + \nu_s \gamma_s \rho_s ds + \nu_s \gamma_s \eta_s \rcor_s dW_s^1 + \nu_s \gamma_s \eta_s \sqrt{1-\rcor_s^2} dW_s^2 \\
			& \quad + \nu_s \gamma_s \eta_s^2 ds + \nu_s \gamma_s \sigma_s \eta_s \rcor_s ds \\
			& = \nu_s \gamma_s \Bigg( \big( \mu_s +\rho_s +\eta_s^2 + \sigma_s \eta_s \rcor_s \big) ds 
			+ \big( \sigma_s + \eta_s \rcor_s \big)dW_s^1 
			+ \eta_s \sqrt{1-\rcor_s^2}dW_s^2 \Bigg) .
		\end{split}
	\end{equation}
	Also by integration by parts, and using \eqref{eq:nuinvdyn}, \eqref{eq:resilience} and \eqref{eq:sqrtalphadyn}, we obtain for all $s \in [t,T]$ that 
	\begin{equation}\label{eq:2008b}
		\begin{split}
			& d(\gamma_s^{-\frac12} \nu_s^{-1}) 
			= -\gamma_s^{-\frac12} \nu_s^{-1} dR_s + \nu_s^{-1} d\gamma_s^{-\frac12} - \nu_s^{-1} d[R,\gamma^{-\frac12}]_s \\
			& = - \gamma_s^{-\frac12} \nu_s^{-1} \rho_s ds - \gamma_s^{-\frac12} \nu_s^{-1} \eta_s \rcor_s dW_s^1 - \gamma_s^{-\frac12} \nu_s^{-1} \eta_s \sqrt{1-\rcor_s^2} dW_s^2 \\
			& \quad + \gamma_s^{-\frac12} \nu_s^{-1} \left( -\frac12\mu_s + \frac38 \sigma_s^2 \right) ds 
			- \frac12 \gamma_s^{-\frac12} \nu_s^{-1}\sigma_s dW_s^1 + \frac12 \gamma_s^{-\frac12} \nu_s^{-1} \sigma_s \eta_s \rcor_s ds \\
			& = \alpha_s \Bigg(  
			\Big( - \rho_s - \frac12 \mu_s + \frac38 \sigma_s^2 + \frac12 \sigma_s \eta_s \rcor_s \Big) ds 
			+ \Big( - \eta_s \rcor_s - \frac12 \sigma_s \Big) dW_s^1
			- \eta_s \sqrt{1-\rcor_s^2} dW_s^2
			\Bigg)  .
		\end{split}
	\end{equation}
	It follows from \eqref{eq:2008a} and \eqref{eq:2008b} for all $s \in [t,T]$ that 
	\begin{equation}\label{eq:2008c}
		\begin{split}
			d[\gamma^{-\frac12}\nu^{-1},\nu\gamma]_s 
			& = \gamma_s^{\frac12} \Big( - \eta_s \rcor_s - \frac12 \sigma_s \Big)\big( \sigma_s + \eta_s \rcor_s \big) ds 
			- \gamma_s^{\frac12} \eta_s^2 (1-\rcor_s^2) ds \\
			& = - \gamma_s^{\frac12} \left( \frac32 \eta_s \sigma_s \rcor_s + \frac12 \sigma_s^2 + \eta_s^2 \right) ds .
		\end{split}
	\end{equation} 
	We then plug \eqref{eq:2008a}, \eqref{eq:2008b} and \eqref{eq:2008c} into \eqref{eq:20080}, which yields \eqref{eq:dynalphabeta}.
\end{proof}

\begin{proof}[Proof of \Cref{lem:scaledhiddendevdyn}]
	We denote $\alpha_s=\gamma_s^{-\frac12} \nu_s^{-1}$, $s\in [t,T]$, and 
	$\beta_s = d - \gamma_t x - \int_t^s X_r d(\nu_r\gamma_r)$, $s \in [t,T]$. 
	It then holds that $\ssfi_s=\alpha_s\beta_s$, $s \in [t,T]$.
	We use \Cref{lem:dynalphabeta} and substitute $-\gamma^{\frac12} X = \ssfi - \gamma^{-\frac12} D$ in \eqref{eq:dynalphabeta} to obtain for all $s \in [t,T]$ that 
	\begin{equation*}
	\begin{split}
	& d\ssfi_s \\
	& = \big( \ssfi_s - \gamma_s^{-\frac12} D_s \big) \Bigg( 
	\left( \mu_s + \rho_s -\frac12 \sigma_s \eta_s \rcor_s - \frac12 \sigma_s^2 \right) 
	ds 
	+ \big( \sigma_s + \eta_s \rcor_s \big)dW_s^1 
	+ \eta_s \sqrt{1-\rcor_s^2}dW_s^2 \Bigg) \\
	& \quad + \ssfi_s \Bigg(  
	\Big( - \rho_s - \frac12 \mu_s + \frac38 \sigma_s^2 + \frac12 \sigma_s \eta_s \rcor_s \Big) ds 
	+ \Big( - \eta_s \rcor_s - \frac12 \sigma_s \Big) dW_s^1
	- \eta_s \sqrt{1-\rcor_s^2} dW_s^2
	\Bigg) \\
	& = 
	- \gamma_s^{-\frac12} D_s \Bigg( 
	\left( \mu_s + \rho_s -\frac12 \sigma_s \eta_s \rcor_s - \frac12 \sigma_s^2 \right) 
	ds 
	+ \big( \sigma_s + \eta_s \rcor_s \big)dW_s^1 
	+ \eta_s \sqrt{1-\rcor_s^2}dW_s^2 \Bigg) \\
	& \quad + 
	\ssfi_s \Bigg( \left( \frac12 \mu_s - \frac18 \sigma_s^2 \right) ds 
	+ \frac12 \sigma_s dW_s^1 \Bigg) . 
	\end{split} 
	\end{equation*}
	This proves the dynamics in~\eqref{eq:dynscaledhiddendev}. 
	
	In particular, $\ssfi$ satisfies an SDE that is linear in $\ssfi$ and $\gamma^{-\frac12}D$. Furthermore, boundedness of $\rho,\mu,\sigma,\eta,\rcor$ implies that the coefficients of the SDE are bounded. Since moreover $E[ \int_t^T \big(\gamma_s^{-\frac12} D_s\big)^2 ds ] <\infty$ by assumption~(A1) and $\ssfi_t=\gamma_t^{-\frac12} d - \gamma_t^{\frac12} x$ (cf.~\eqref{eq:scaledhiddendevdef}) is square integrable, we have that $E[ \sup_{s \in [t,T]} \ssfi_s^2 ]<\infty$  (see, e.g., \cite[Theorem~3.2.2 and Theorem~3.3.1]{zhang}).	
	
	We next prove that cost functional~\eqref{eq:defcostfctpm} admits representation~\eqref{eq:quadr_pmj}. 
	To this end, note that by \eqref{eq:scaledhiddendevdef} it holds for all $s \in [t,T]$ that 
	\begin{equation*}
		\gamma_s \!\left(X_s\! - \zeta_s \right)^2 
		 = 
		\left(\gamma^{-\frac12}_s D_s - \ssfi_s - \gamma_s^{\frac12} \zeta_s \right)^{\!2} \!
		= \gamma_s^{-1}D_s^2 
		- 2 \gamma_s^{-\frac12}D_s \left( \ssfi_s + \gamma_s^{\frac12} \zeta_s \right) + \left( \ssfi_s + \gamma_s^{\frac12} \zeta_s \right)^{\!2} \! . 
	\end{equation*}
	Due to assumption~\eqref{eq:conditionxi} on $\zeta$ and $E[\sup_{s\in[t,T]} \ssfi_s^2 ]<\infty$, we have that $E_t[\int_t^T (\ssfi_s + \gamma_s^{\frac12} \zeta_s )^2 ds]<\infty$. 
	This, assumption (A1), and the Cauchy--Schwarz inequality imply that also $E_t[\int_t^T \lvert \gamma_s^{-\frac12} D_s (\ssfi_s + \gamma_s^{\frac12} \zeta_s) \rvert ds]<\infty$. 
	Since $\lambda$ is bounded, we conclude that 
	\begin{equation}\label{eq:risktermrewritten}
		\begin{split}
			E_t\left[ \int_t^T \lambda_s \gamma_s \left(X_s - \zeta_s \right)^2  ds \right] 
			& = E_t\left[ \int_t^T \lambda_s \gamma_s^{-1}D_s^2 ds \right] 
			+ E_t\left[ \int_t^T \lambda_s \left( \ssfi_s + \gamma_s^{\frac12} \zeta_s \right)^2 ds \right] \\
			& \quad - 2 E_t\left[ \int_t^T \lambda_s \gamma_s^{-\frac12}D_s \left( \ssfi_s + \gamma_s^{\frac12} \zeta_s \right) ds \right] ,
		\end{split}
	\end{equation}
	where all conditional expectations are well-defined and finite. 
	Moreover, \eqref{eq:scaledhiddendevdef} implies that 
	$\gamma_T^{-\frac12} D_T = \ssfi_T + \gamma_T^{\frac12} X_T$, and thus 
	$\gamma_T^{-1}D_T^2=(\ssfi_T + \sqrt{\gamma_T}\xi)^2$. 
	Inserting this and~\eqref{eq:risktermrewritten} into~\eqref{eq:defcostfctpm}, we obtain~\eqref{eq:quadr_pmj}.
\end{proof}

\begin{lemma}\label{lem:ismetric}
	Let $t \in [0,T]$ and $x,d \in \R$. Then, \eqref{eq:defmetriconpm} defines a metric on $\cA_t^{pm}(x,d)$ (identifying any processes that are equal $dP\times ds|_{[t,T]}$-a.e.).
\end{lemma}

\begin{proof}	
	Note first that it holds for all $X,Y \in \cA_t^{pm}(x,d)$ that $\md(X,Y)\geq 0$, and that $\md(X,Y)$ is finite due to (A1). 
	Symmetry of $\md$ is obvious. 
	The triangle inequality follows from the Cauchy--Schwarz inequality. 

	Let $X,Y \in \cA_t^{pm}(x,d)$ with associated deviation processes $D^X, D^Y$.
	
	If $X=Y$ $dP\times ds|_{[t,T]}$-a.e., then $\gamma^{-\frac12}D^X=\gamma^{-\frac12}D^Y$ $dP\times ds|_{[t,T]}$-a.e., and thus $\md(X,Y)=(E[\int_t^T (\gamma_s^{-\frac12}D_s^X - \gamma_s^{-\frac12}D_s^Y)^2 ds ])^{\frac12} = 0$. 
	
	For the other direction, suppose that $\md(X,Y)=0$. 
	This implies that $\gamma^{-\frac12} D^X - \gamma^{-\frac12} D^Y = 0$ $dP\times ds|_{[t,T]}$-a.e. 
	By definition of $D^X$ and $D^Y$ it further follows from a multiplication by $\nu\gamma^{\frac12}$ that 
	$\nu_s \gamma_s (X_s-Y_s) = \int_t^{s} (X_r - Y_r) d(\nu_r \gamma_r)$ $\,dP\times ds|_{[t,T]}\text{-a.e.}$ 
	Observe that $\nu\gamma>0$ and consider the stochastic integral equation 
	\begin{equation}\label{eq:3008a}
		K_s=\int_t^s K_r\nu_r^{-1}  \gamma_r^{-1} d(\nu_r\gamma_r), \quad s\in[t,T].
	\end{equation}
	Define $L=(L_s)_{s \in [0,T]}$ by $L_0=0$, 
	\begin{equation*}
		dL_s = \big( \mu_s +\rho_s +\eta_s^2 + \sigma_s \eta_s \rcor_s \big) ds 
		+ \big( \sigma_s + \eta_s \rcor_s \big)dW_s^1 
		+ \eta_s \sqrt{1-\rcor_s^2}dW_s^2, 
		\quad s \in [0,T]. 
	\end{equation*}
	It then follows from \eqref{eq:2008a} that \eqref{eq:3008a} can be written as 
	$K_s=\int_t^s K_r dL_r$, $s\in[t,T]$. 
	This has the unique solution $K=0$. 
	We therefore conclude that $X=Y$ $dP\times ds|_{[t,T]}$-a.e.
\end{proof}

We now prepare the proof of \Cref{thm:contextcostfct}. 
The next result on the scaled hidden deviation is helpful in \Cref{thm:contextcostfct} 
in order to show convergence of the cost functional.

\begin{lemma}\label{lem:scaledhiddendevconv}
		Let $t \in [0,T]$, $x,d \in \R$, and $X \in \cA^{pm}_t(x,d)$ with associated deviation~$D$ and scaled hidden deviation $\ssfi$. 
		Suppose in addition that $(X^n)_{n \in \N}$ is a sequence in $\cA^{pm}_t(x,d)$ such that 
		$\lim_{n\to\infty} E[ \int_t^T (D_s^n - D_s)^2 \gamma_s^{-1} ds ] = 0$. 
		for the associated deviation processes $D^n$, $n \in \N$. It then holds for the associated scaled hidden deviation processes $\ssfi^n$, $n\in\N$, that 
		$\lim_{n\to\infty} E[ \sup_{s \in [t,T]} ( \ssfi_s^n - \ssfi_s )^2 ] = 0.$
\end{lemma}

\begin{proof}
	Define $\delta \ssfi^n = \ssfi^n - \ssfi$, $n \in\N$, and let for $n\in\N$, $s \in [t,T]$, $z\in\R$ 
	\begin{equation*}
		\begin{split}
			b^n_s(z) & = -\frac12 \Big( 2(\rho_s+\mu_s) - \sigma_s^2 - \sigma_s\eta_s\rcor_s \Big) \big(\gamma^{-\frac12}_s D^n_s - \gamma^{-\frac12}_s D_s \big)  
			+ \frac12 \left( \mu_s - \frac14 \sigma_s^2 \right) z ,\\
			a_s^n(z) & = 
			\bigg( - (\sigma_s + \eta_s \rcor_s) \big(\gamma^{-\frac12}_s D^n_s - \gamma^{-\frac12}_s D_s \big) + \frac12 \sigma_s z, - \eta_s \sqrt{1-\rcor_s^2} \big(\gamma^{-\frac12}_s D^n_s - \gamma^{-\frac12}_s D_s \big) \bigg).
		\end{split}
	\end{equation*}
	In view of \eqref{eq:dynscaledhiddendev} it then holds for all $n\in\N$ that 
	\begin{equation*}
		d(\delta \ssfi^n_s) = b_s^n(\delta\ssfi_s^n) ds + a_s^n(\delta \ssfi_s^n) d
		\begin{pmatrix}
			W_s^1\\
			W_s^2 
		\end{pmatrix}, 
		\quad s \in [t,T], \quad \delta\ssfi_t^n=0.
	\end{equation*}
	Linearity of $b^n$, $a^n$, $n\in\N$, and boundedness of $\mu,\rho,\sigma,\eta,\rcor$ imply that there exists $c_1 \in (0,\infty)$ such that 
	for all $n\in\N$ and all $z_1,z_2 \in \R$ it holds $dP\times ds|_{[t,T]}$-a.e.\ that
	\begin{equation*}
		\lvert b^n(z_1) - b^n(z_2) \rvert + \lVert a^n(z_1) - a^n(z_2) \rVert_2 \leq \frac12 \left\lvert \mu - \frac14 \sigma^2 \right\rvert \lvert z_1-z_2 \rvert + \frac12 \lvert \sigma\rvert \lvert z_1-z_2\rvert 
		\leq c_1 \lvert z_1-z_2\rvert .
	\end{equation*} 
	By boundedness of $\mu,\rho,\sigma,\eta,\rcor$ and Jensen's inequality, we have some $c_2 \in (0,\infty)$ such that for all $n \in \N$,
	\begin{equation*}
		E\left[ \left( \int_t^T \lvert b_s^n(0) \rvert ds \right)^2 \right] 
		+ E\left[ \int_t^T \lVert a^n_s(0) \rVert_2^2 ds \right] \leq c_2 E\left[ \int_t^T (D_s^n-D_s)^2 \gamma_s^{-1} ds \right].
	\end{equation*}
	E.g., \cite[Theorem~3.2.2]{zhang} (see also \cite[Theorem~3.4.2]{zhang}) now implies that 
	there exists $c_3 \in (0,\infty)$ such that for all $n\in\N$ 
	\begin{equation*}
		\begin{split}
			E\left[ \sup_{s\in[t,T]} \lvert \ssfi_s^n - \ssfi_s \rvert^2 \right] 
			& \leq c_3 E\left[ \left( \int_t^T \lvert b_s^n(0) \rvert ds \right)^2 + \int_t^T \lVert a^n_s(0) \rVert_2^2 ds \right] \\
			& \leq c_2 c_3  E\left[ \int_t^T (D_s^n-D_s)^2 \gamma_s^{-1} ds \right] .
		\end{split}
	\end{equation*} 
	The claim follows from the assumption that $\lim_{n\to\infty} E[ \int_t^T \left( D_s^n - D_s \right)^2 \gamma_s^{-1} ds ] = 0$. 
\end{proof}

In order to establish existence of an appropriate approximating sequence in \Cref{thm:contextcostfct}, we rely on \Cref{lem:approxargumentKS} below. 
For its statement and the proof of the second part of  \Cref{thm:contextcostfct}, 
we introduce a process $Z=(Z_s)_{s \in [0,T]}$ defined by 
\begin{equation}\label{eq:defZ}
	Z_s = \exp\left( -\int_0^s \left(\frac12 \sigma_r + \eta_r \rcor_r \right) dW^1_r - \int_0^s \eta_r \sqrt{1-\rcor_r^2} dW_r^2 \right), \quad s\in[0,T].
\end{equation}
Observe that by  It\^o's lemma, $Z$ solves the SDE 
\begin{equation}\label{eq:SDEforZ}
	\begin{split}
		& dZ_s =  \frac{Z_s}2 \left( \left( \frac12 \sigma_s + \eta_s \rcor_s \right)^2 + \eta_s^2 (1-\rcor_s^2) \right) ds 
		- Z_s \left( \frac12 \sigma_s + \eta_s \rcor_s, \eta_s \sqrt{1-\rcor_s^2} \right) d \begin{pmatrix}
			W^1_s\\
			W_s^2 
		\end{pmatrix}, \\
		& s \in [0,T], \quad Z_0=1.
	\end{split}
\end{equation}

\begin{lemma}\label{lem:approxargumentKS}
	Let $t \in [0,T]$ and let 
	$u=(u_s)_{s \in [t,T]} \in \cL_t^2$. 
	Then there exists a sequence of bounded c\`adl\`ag finite variation processes $(v^n)_{n\in\N}$ such that 
	\begin{equation*}
		\lim_{n\to\infty} E\left[ \int_t^T \left( \frac{u_s}{Z_s} - v_s^n \right)^2 Z_s^2 ds \right] = 0.
	\end{equation*}
	In particular, for the sequence of processes $(u^n)_{n\in\N}$ defined by $u^n = v^n Z$, $n \in \N$, it holds for all $n \in \N$ that $u^n$ is a c\`adl\`ag semimartingale and $E[\sup_{s\in[t,T]} \lvert u_s^{n} \rvert^p]<\infty$ for any $p\geq 2$ (in particular, $u^n \in \cL_t^2$), and that 
	$	\lim_{n\to\infty} E[ \int_t^T \left( u_s - u_s^n \right)^2 ds ] = 0.$
\end{lemma}

\begin{proof}
	Define $A=(A_s)_{s\in[0,T]}$ by $A_s=\int_0^s Z_r^2 dr$, $s\in [0,T]$. 
	Moreover, let $v=(v_s)_{s \in [t,T]}$ be defined by  $v_s=\frac{u_s}{Z_s}$, $s \in [t,T]$. 
	We verify the assumptions of Lemma~2.7 in Section~3.2 of \cite{karatzasshreve}. 
	The process $A$ is continuous, adapted and nondecreasing. 
	Note that boundedness of $\sigma$, $\eta$ and $\rcor$ implies that the coefficients of \eqref{eq:SDEforZ} are bounded. 
	It follows for any $p\geq 2$ that $E[\sup_{s\in[0,T]} \lvert Z_s \rvert^p]<\infty$ (see, e.g., \cite[Theorem~3.4.3]{zhang}), and hence 
	$E[A_T]=E[\int_0^T Z_r^2 dr]<\infty$. 
	Since $u \in \cL_t^2$, we have that $v$ is progressively measurable and satisfies $E[\int_t^T v_s^2 dA_s]=E[\int_t^T u_s^2 ds]< \infty$. 
	Thus, Lemma 2.7 in Section 3.2 of \cite{karatzasshreve} applies and yields that there exists a sequence $(\hat v^n)_{n\in\N}$ of (c\`agl\`ad) simple processes $\hat v^n=(\hat v_s^n)_{s \in [t,T]}$, $n \in\N$, such that $\lim_{n\to\infty} E[\int_t^T (v_s - \hat v_s^n)^2 dA_s] = 0$. 
	Define $v_s^n(\omega)=\lim_{r \downarrow s} \hat v_r^n(\omega)$, $s \in [t,T)$, $\omega \in \Omega$, $n\in\N$, and $v_T^n=0$, $n\in\N$. 
	Then, $(v^n)_{n\in\N}$ is a sequence of bounded c\`adl\`ag finite variation processes such that $\lim_{n\to\infty}E[\int_t^T (v_s - v_s^n)^2 dA_s] = 0$. 
	Note that for each $n\in\N$, $u^{n}=(u_s^{n})_{s \in [t,T]}$ defined by $u^{n}_s=v^{n}_s Z_s$, $s\in[t,T]$, is 
	c\`adl\`ag. 
	Since $v^{n}$ is bounded for all $n \in \N$ and $E[\sup_{s\in[0,T]} \lvert Z_s \rvert^p]<\infty$ for any $p\geq 2$, we have that $E[ \sup_{s \in [t,T]} \lvert u_s^{n} \rvert^p ]$ is finite for all $n \in \N$ and any $p\geq 2$. 
	It furthermore holds that 
	$E[\int_t^T (u_s - u_s^{n} )^2 ds] = E[\int_t^T (v_s - v_s^{n})^2 dA_s]\to 0$ as $n \to \infty$. 
\end{proof}

For the part in \Cref{thm:contextcostfct} on completeness of $(\cA_t^{pm}(x,d),\md)$ we show how to construct an execution strategy $X^0 \in \cA_t^{pm}(x,d)$ based on a square integrable process $u^0$ and a process $H^0$ that satisfies SDE \eqref{eq:dynscaledhiddendev} (with $u^0$ instead of $\gamma^{-\frac12}D$). 
This result is also crucial for \Cref{propo:givenugetX}.

\begin{lemma}\label{lem:getXfromu}
	Let $t \in [0,T]$ and $x,d \in \R$. 
	Suppose that $u^0=(u^0_s)_{s\in[t,T]} \in \cL_t^2$, and  
	let $H^0=(H^0_s)_{s\in[t,T]}$ be given by $H^0_t=\frac{d}{\sqrt{\gamma_t}} - \sqrt{\gamma_t}x$, 
	\begin{equation}\label{eq:SDEH0}
		\begin{split}
			dH^0_s & = \left(\frac12 \left( \mu_s - \frac14 \sigma_s^2 \right) H^0_s -\frac12 \left( 2(\rho_s+\mu_s) - \sigma_s^2 - \sigma_s\eta_s\rcor_s \right) u^0_s \right)ds\\
			&\quad +\left(\frac12 \sigma_s H^0_s - (\sigma_s + \eta_s \rcor_s) u^0_s\right) dW^1_s
			- \eta_s \sqrt{1-\rcor_s^2} u^0_s dW_s^2, \quad s\in [t,T].
		\end{split}
	\end{equation}
	Define $X^0=(X^0_s)_{s\in[t-,T]}$ by 
	$X^0_s=\gamma_s^{-\frac12}(u^0_s-H^0_s)$, $s \in[t,T)$, $X^0_{t-}=x$, $X^0_T=\xi$. 
	Then, $X^0 \in \cA_t^{pm}(x,d)$, and for the associated deviation process $D^0=(D^0_s)_{s \in [t-,T]}$ it holds $D^0=\gamma X^0 + \gamma^{\frac12}H^0$.
\end{lemma}

\begin{proof}
	First, $X^0$ is progressively measurable and has initial value $X^0_{t-}=x$ and terminal value $X^0_T=\xi$. 
	Furthermore, it holds that 
	\begin{equation*}
		\begin{split}
			\int_t^T (X_s^0)^2 ds 
			& \leq 2 \int_t^T \gamma_s^{-1}(u_s^0)^2 ds + 2 \int_t^T \gamma_s^{-1} (H^0_s)^2 ds < \infty \text{ a.s.}
		\end{split}
	\end{equation*}
	since $\gamma$ and $H^0$ have a.s.\ continuous paths and $E[\int_t^T (u_s^0)^2 ds]<\infty$. 
	We are therefore able to define $D^0$ by \eqref{eq:defdeviationpm}. 		
	Moreover, denote $\alpha_s=\gamma_s^{-\frac12} \nu_s^{-1}$, $s\in [t,T]$, and 
	$\beta_s = d - \gamma_t x - \int_t^s X^0_r d(\nu_r\gamma_r)$, $s \in [t,T]$. 
	It follows from \Cref{lem:dynalphabeta} and $-\gamma_s^{\frac12} X^0_s = H^0_s - u^0_s$, $s \in [t,T)$, that for all $s \in [t,T]$ 
	\begin{equation*}
		\begin{split}
			& d(\alpha_s\beta_s) \\ 
			& = 
			(H^0_s - u^0_s) \Bigg( \Big( \mu_s +\rho_s - \frac12 \sigma_s \eta_s \rcor_s - \frac12 \sigma_s^2 \Big) ds 
			+ \big( \sigma_s + \eta_s \rcor_s \big)dW_s^1 
			+ \eta_s \sqrt{1-\rcor_s^2}dW_s^2 \Bigg) \\
			& \quad + \alpha_s\beta_s  \Bigg(  
			\Big( - \rho_s - \frac12 \mu_s + \frac38 \sigma_s^2 + \frac12 \sigma_s \eta_s \rcor_s \Big) ds 
			+ \Big( -\eta_s \rcor_s - \frac12 \sigma_s \Big) dW_s^1
			- \eta_s \sqrt{1-\rcor_s^2} dW_s^2
			\Bigg) .
		\end{split}
	\end{equation*}
	We combine this with  
	\begin{equation*}
		\begin{split}
			dH^0_s & = -u^0_s \Bigg( \Big( \mu_s + \rho_s - \frac12 \sigma_s \eta_s \rcor_s - \frac12 \sigma_s^2 \Big) ds 
			+ (\sigma_s + \eta_s \rcor_s) dW_s^1 
			+  \eta_s \sqrt{1-\rcor_s^2} dW_s^2 \Bigg) \\
			& \quad + H^0_s \Bigg( \Big( \frac12 \mu_s - \frac18 \sigma_s^2 \Big) ds + \frac12 \sigma_s dW_s^1 \Bigg)
			, \quad s\in [t,T],
		\end{split}
	\end{equation*} 
	to obtain for all $s \in [t,T]$ that 
	\begin{equation}\label{eq:SDEalhpabetaminusssfiu}
		\begin{split}
			d(\alpha_s\beta_s - H^0_s) 
			& = (\alpha_s\beta_s - H^0_s) 
			\Bigg(  
			\Big( - \rho_s - \frac12 \mu_s + \frac38 \sigma_s^2 + \frac12 \sigma_s \eta_s \rcor_s \Big) ds 
			+ \Big( -\eta_s \rcor_s - \frac12 \sigma_s \Big) dW_s^1 \\
			& \quad \quad \quad \qquad \qquad- \eta_s \sqrt{1-\rcor_s^2} dW_s^2
			\Bigg) .
		\end{split}
	\end{equation}
	Note that $\alpha_t\beta_t=\gamma_t^{-\frac12} d - \gamma_t^{\frac12} x = H^0_t$. 
	We thus conclude that $0$ is the unique solution of~\eqref{eq:SDEalhpabetaminusssfiu}, and hence 
	$H^0_s = \gamma^{-\frac12}_s \nu_s^{-1} (d-\gamma_t x - \int_t^s X^0_r d(\nu_r\gamma_r) )$, $s\in[t,T]$. 
	This implies that $D^0=\gamma X^0 + \gamma^{\frac12} H^0$, i.e., $D^0_s = \gamma_s^{\frac12} u^0_s$, $s \in [t,T)$, and $D^0_T=\gamma_T \xi + \gamma_T^{\frac12} H^0_T$.
	The fact that $E[ \int_t^T (u_s^0)^2 ds ] < \infty$ then immediately yields that (A1) holds. 
	This proves that $X^0 \in \cA_t^{pm}(x,d)$.
\end{proof}

We finally are able to prove \Cref{thm:contextcostfct}.

\begin{proof}[Proof of \Cref{thm:contextcostfct}]
	(i) 
	Denote by $D$, $D^n$, $n \in\N$, the deviation processes associated to $X$, $X^n$, $n\in\N$, and 
	let $\ssfi$ and $\ssfi^n$, $n\in\N$, be the scaled hidden deviation processes. 	
	By~\Cref{lem:scaledhiddendevdyn} 
	it holds for all $n \in\N$ that 
	\begin{equation*}
		\begin{split}
			\left\lvert  \pmJ_t(x,d,X^{n}) - \pmJ_t(x,d,X) \right\rvert 
			& = \Bigg \lvert  
			\frac12 E_t\left[ \int_t^T \gamma_s^{-1} \left( (D_s^{n})^2 - D_s^2 \right) 
			2(\kappa_s+\lambda_s) ds \right] \\
			& \quad - 2 E_t\!\left[ \int_t^T \! \lambda_s \gamma_s^{-\frac12} \! \left( D_s^{n} \big( \ssfi_s^n + \gamma_s^{\frac12} \zeta_s \big) \! - \! D_s \big( \ssfi_s + \gamma_s^{\frac12} \zeta_s \big) \! \right)\! ds \right] \\
			& \quad + E_t\left[ \int_t^T \lambda_s \left( \big( \ssfi_s^n + \gamma_s^{\frac12} \zeta_s \big)^2 - \big( \ssfi_s + \gamma_s^{\frac12} \zeta_s \big)^2 \right) ds \right] \\
			& \quad + \frac12 E_t\left[ ( \ssfi_T^n + \gamma_T^{\frac12} \xi )^2 - ( \ssfi_T + \gamma_T^{\frac12} \xi )^2 \right] 
			\Bigg\rvert .
		\end{split}
	\end{equation*}
	Boundedness of $\lambda,\rcor,\rho,\mu,\eta$ and $\sigma$ implies 
	(recall also~\eqref{eq:defkappa}) 
	that there exists some $c\in (0,\infty)$ such that for all $n \in \N$ it holds that 
	\begin{equation}\label{eq:convcostfctintermediatestep}
		\begin{split}
			& E\left[ \left\lvert  \pmJ_t(x,d,X^{n}) - \pmJ_t(x,d,X) \right\rvert \right] \\
			& \leq 
			E\left[ \left\lvert ( \ssfi_T^n + \gamma_T^{\frac12} \xi )^2 - ( \ssfi_T + \gamma_T^{\frac12} \xi )^2 \right\rvert  \right] 
			+
			c  E\left[ \int_t^T \left\lvert \gamma_s^{-1} \left( (D_s^{n})^2 - D_s^2 \right) \right\rvert ds \right] \\
			& \quad + c  E\left[ \int_t^T \left\lvert  \gamma_s^{-\frac12} \left( D_s^{n} \big( \ssfi_s^n + \gamma_s^{\frac12} \zeta_s \big) - D_s \big( \ssfi_s + \gamma_s^{\frac12} \zeta_s \big) \right) \right\rvert ds \right]   \\
			& \quad  
			+ c E\left[ \int_t^T  \left\lvert \big( \ssfi_s^n + \gamma_s^{\frac12} \zeta_s \big)^2 - \big( \ssfi_s + \gamma_s^{\frac12} \zeta_s \big)^2 \right\rvert ds \right] 
			.
		\end{split}
	\end{equation}
	We treat the terminal costs first. 
	It holds for all $n \in \N$ that 
	\begin{equation*}
		\begin{split}
			& E\left[ \left\lvert ( \ssfi_T^n + \gamma_T^{\frac12} \xi )^2 - ( \ssfi_T + \gamma_T^{\frac12} \xi )^2  \right\rvert \right] 
			= E\left[ \left\lvert (\ssfi_T^n)^2 + 2\ssfi_T^n \gamma_T^{\frac12} \xi - \ssfi_T^2 - 2 \ssfi_T \gamma_T^{\frac12} \xi \right\rvert \right] \\  
			& \qquad\qquad\qquad\qquad \leq  E\left[ \left\lvert (\ssfi_T^{n})^2 - \ssfi_T^2 \right\rvert \right]  
			+ 2 E\left[ \left\lvert (\ssfi^{n}_T - \ssfi_T) \gamma_T^{\frac12} \xi \right\rvert \right] \\
			& \qquad\qquad\qquad\qquad \leq E\left[ \left\lvert (\ssfi_T^{n})^2 - \ssfi_T^2 \right\rvert \right]  
			+ 2 \left( E\left[ (\ssfi^{n}_T - \ssfi_T)^2 \right] \right)^{\frac12} 
			\left( E\left[\gamma_T \xi^2 \right] \right)^{\frac12} .
		\end{split}
	\end{equation*}
	From 
	\begin{equation}\label{eq:assconvDn}
		\lim_{n\to\infty} E\left[\int_t^T (D_s^n - D_s)^2 \gamma_s^{-1} ds \right] = 0
	\end{equation} 
	(cf. \eqref{eq:defmetriconpm}) 
	and \Cref{lem:scaledhiddendevconv} 
	we have that 
	\begin{equation}\label{eq:convsupHn}
		\lim_{n\to\infty} E\left[ \sup_{s \in [t,T]} \lvert \ssfi_s^{n} - \ssfi_s \rvert^2 \right] = 0 .
	\end{equation}
	Since furthermore $E[\gamma_T\xi^2]<\infty$, we obtain that 
	$\lim_{n\to\infty} E[ \lvert  ( \ssfi_T^n + \gamma_T^{\frac12} \xi )^2 - ( \ssfi_T + \gamma_T^{\frac12} \xi )^2 \rvert ]  = 0.$ 
	The second term in~\eqref{eq:convcostfctintermediatestep} converges to $0$ using \eqref{eq:assconvDn}. 
	For the third term in \eqref{eq:convcostfctintermediatestep} we have for all $n \in \N$ that 
	\begin{equation}\label{eq:crosstermconv}
		\begin{split}
			& E\left[ \int_t^T \left\lvert  \gamma_s^{-\frac12} \left( D_s^{n} \big( \ssfi_s^n + \gamma_s^{\frac12} \zeta_s \big) - D_s \big( \ssfi_s + \gamma_s^{\frac12} \zeta_s \big) \right) \right\rvert ds \right]   \\ 
			& \leq E\left[ \int_t^T 
			\left( \big\lvert \ssfi_s + \gamma_s^{\frac12} \zeta_s \big\rvert \,  \lvert D_s^{n} - D_s \rvert \gamma_s^{-\frac12} +  \gamma_s^{-\frac12} \lvert D_s^{n} \rvert \, \lvert \ssfi_s^{n} - \ssfi_s \rvert \right)
			ds \right] \\
			& \leq \left(E\bigg[  \int_t^T \big( \ssfi_s +  \gamma_s^{\frac12} \zeta_s )^2 ds \bigg]\right)^{\frac12} \left(E\bigg[ \int_t^T (D_s^{n} - D_s )^2 \gamma_s^{-1} ds\bigg]\right)^{\frac12} \\
			&\quad\,
			+ \left(E\bigg[ \int_t^T \gamma_s^{-1} (D_s^{n})^2 ds\bigg]\right)^{\frac12} T^{\frac12} \left(E\bigg[ \sup_{s\in[t,T]} \lvert \ssfi_s^{n} - \ssfi_s \rvert^2 \bigg]\right)^{\frac12} . 
		\end{split}
	\end{equation}
	By \Cref{lem:scaledhiddendevdyn} and \eqref{eq:conditionxi} it holds that 
	$E[  \int_t^T ( \ssfi_s +  \gamma_s^{\frac12} \zeta_s )^2 ds ] < \infty$.
	Moreover, due to \eqref{eq:assconvDn}, we have that $E[ \int_t^T \gamma_s^{-1}(D_s^n)^2 ds]$ is uniformly bounded in $n\in\N$. 
	It thus follows from \eqref{eq:assconvDn}, \eqref{eq:convsupHn} and \eqref{eq:crosstermconv} that the third term in~\eqref{eq:convcostfctintermediatestep} converges to $0$ as $n\to\infty$. 
	The last term in~\eqref{eq:convcostfctintermediatestep} converges to $0$ using \eqref{eq:conditionxi} and \eqref{eq:convsupHn}. 
	This proves claim (i).
	
	(ii) 
	Suppose that $X\in\cA^{pm}_t(x,d)$. 
	Let $u=(u_s)_{s\in[t,T]}$ be defined by  $u_s=\gamma_s^{-\frac12} D_s$, $s \in [t,T]$, where $D$ denotes the deviation associated to $X$. 
	Then, $u$ is a progressively measurable process, and due to assumption (A1) it holds that $E[\int_t^T u_s^2 ds]<\infty$. 
	By~\Cref{lem:approxargumentKS} there exists a sequence of bounded c\`adl\`ag finite variation processes $(v^n)_{n\in\N}$ such that 
	$\lim_{n\to\infty} E[ \int_t^T ( \frac{u_s}{Z_s} - v_s^n )^2 Z_s^2 ds ] = 0$, where $Z$ is defined in~\eqref{eq:defZ}. 
	Set $u^n = v^n Z$, $n \in \N$. This is a sequence of c\`adl\`ag semimartingales in $\cL_t^2$ that satisfies 
	$\lim_{n\to\infty} \lVert u-u^n \rVert_{\cL_t^2} = 0$. 
	Moreover, it holds for all $n \in \N$ and any $p\geq 2$ that $E[\sup_{s\in[t,T]} \lvert u_s^n \rvert^p ]<\infty$. 
	For each $u^n$, $n\in\N$, let $H^n=(H^n_s)_{s\in[t,T]}$ be the solution of~\eqref{eq:SDEH0}. 
	We then define a sequence of c\`adl\`ag semimartingales $X^n=(X^n_s)_{s\in[t-,T]}$, $n \in \N$, by 
	$X^n_s=\gamma_s^{-\frac12}(u_s^n-H_s^n)$, $s\in [t,T)$, $X_{t-}^n=x$, $X_T^n=\xi$. 
	By \Cref{lem:getXfromu} we have for all $n\in\N$ that $X^n\in\cA_t^{pm}(x,d)$ and that $D^n=\gamma X^n + \gamma^{\frac12} H^n$ for the associated deviation process $D^n=(D^n_s)_{s\in[t-,T]}$.
	It follows for all $n\in\N$ that $D_s^n=\gamma_s^{\frac12}u_s^n$, $s\in[t,T)$. 
	Therefore, it holds for all $n\in\N$ that 
	\begin{equation*}
		\md(X^n,X) 
		= \left(E\left[ \int_t^T (D_s^n-D_s)^2 \gamma_s^{-1} ds  \right]\right)^{\frac12} 
		= \left(E\left[ \int_t^T (u_s^n-u_s)^2 ds  \right]\right)^{\frac12} .
	\end{equation*}
	Due to $\lim_{n\to\infty} \lVert u-u^n \rVert_{\cL_t^2} = 0$, we thus have that $\lim_{n\to\infty} \md(X^n,X)=0$. 
	We next show that for all $n\in\N$, $X^n$ has finite variation. 
	To this end, we observe that for all $n\in\N$ and $s\in[t,T)$ it holds by integration by parts that
	\begin{equation}\label{eq:2502a4}
		dX_s^n = \gamma_s^{-\frac12} d(u_s^n-H_s^n) 
		+ (u_s^n-H_s^n)d\gamma_s^{-\frac12} 
		+ d[\gamma^{-\frac12},u^n-H^n]_s.	
	\end{equation}
	Again by integration by parts, and using~\eqref{eq:SDEforZ}, we have for all $n\in\N$ and $s \in [t,T]$ that 
	\begin{equation*}
		\begin{split}
			du_s^n & = v_s^n dZ_s + Z_s dv_s^n + d[v^n,Z]_s \\
			& = \! \frac12 u_s^n \bigg( \!\Big( \frac12 \sigma_s \! + \! \eta_s \rcor_s\! \Big)^{\!2} \!\! +\! \eta_s^2(1-\rcor_s^2) \! \bigg)\! ds 
			- \! u_s^n \Big( \frac12 \sigma_s\! +\! \eta_s \rcor_s\! \Big)\! dW_s^1 \!
			- \! u_s^n \eta_s \sqrt{1\!-\!\rcor_s^2} dW_s^2  \!
			+ \! Z_s dv_s^n .
		\end{split}
	\end{equation*}
	This and~\eqref{eq:SDEH0} yield for all $n\in\N$ and $s\in[t,T]$ that 
	\begin{equation}\label{eq:2502a2}
		\begin{split}
			\gamma_s^{-\frac12}d(u_s^n-H_s^n) & = 
			\gamma_s^{-\frac12} \left( \rho_s + \mu_s + \frac12 \eta_s^2 - \frac38 \sigma_s^2 \right) u_s^n ds 
			- \gamma_s^{-\frac12} \left( \frac12 \mu_s - \frac18 \sigma_s^2 \right) H_s^n ds \\
			& \quad + \gamma_s^{-\frac12} \frac12 \sigma_s (u_s^n-H_s^n) dW_s^1 
			+ \gamma_s^{-\frac12} Z_s dv_s^n . 
		\end{split}
	\end{equation}
	Moreover, it follows from \eqref{eq:sqrtalphadyn} for all $n\in\N$ and $s\in[t,T]$ that 
	\begin{equation}\label{eq:2502a3}
		\begin{split}
			(u_s^n-H_s^n)d\gamma_s^{-\frac12} 
			& = (u_s^n-H_s^n)\gamma_s^{-\frac12} \left( -\frac12 \mu_s + \frac38 \sigma_s^2 \right) ds 
			- (u_s^n - H_s^n) \gamma_s^{-\frac12} \frac12\sigma_s dW_s^1.
		\end{split}
	\end{equation}
	We combine \eqref{eq:2502a4}, \eqref{eq:2502a2}, and \eqref{eq:2502a3} to obtain for all $n\in\N$ and $s\in(t,T)$ that 
	\begin{equation*}
		\begin{split}
			dX_s^n & = \gamma_s^{-\frac12} u_s^n \left( \rho_s + \frac12 \mu_s + \frac12 \eta_s^2  \right) ds 
			- \gamma_s^{-\frac12} H_s^n \frac14 \sigma_s^2 ds 
			+ \gamma_s^{-\frac12} Z_s dv_s^n + d[\gamma^{-\frac12},u^n-H^n]_s.
		\end{split}
	\end{equation*}
	Since $v^n$ has finite variation for all $n\in\N$, this representation shows that also $X^n$ has finite variation for all $n\in\N$.
	Note that for all $n\in\N$, by \Cref{propo:costfunctionalpart}, the process \eqref{eq:deviationdyndR} associated to the c\`adl\`ag finite variation process $X^n$ is nothing but $D^n$. 
	Since $\eta$ is bounded, there exists $c \in (0,\infty)$ such that for all $n\in\N$
	\begin{equation*}
				\begin{split}
						E\left[ \left( \int_t^T (D_s^n)^4 \gamma_s^{-2} \eta_s^2 ds \right)^{\frac12} \right] 
						& = E\left[ \left( \int_t^T (u_s^n)^4 \eta_s^2 ds \right)^{\frac12} \right] 
						\leq c E[ \sup_{s\in[t,T]} (u_s^n)^2 ] < \infty.
					\end{split}
			\end{equation*}
	This implies (A2). Similarly, by boundedness of $\sigma$, we obtain (A3). 
	We thus conclude that $X^n \in \cA_t^{fv}(x,d)$ for all $n \in \N$.

	(iii)
	Let $(X^n)_{n\in\N}$ be a Cauchy sequence in $(\cA_t^{pm}(x,d),\md)$. 
	For $n\in\N$ we denote by $D^n$ the deviation process associated to $X^n$. 
	It then holds that $(\gamma^{-\frac12} D^n)_{n\in\N}$ is a Cauchy sequence in $(\cL_t^2,\lVert \cdot \rVert_{\cL_t^2})$. 
	Since $(\cL_t^2,\lVert \cdot \rVert_{\cL_t^2})$ is complete (see, e.g., Lemma 2.2 in Section 3.2 of \cite{karatzasshreve}), there exists $u^0 \in \cL_t^2$ such that $\lim_{n\to\infty} \lVert \gamma^{-\frac12} D^n - u^0 \rVert_{\cL_t^2} = 0$. 
	Define $X^0=(X^0_s)_{s\in[t-,T]}$ by $X_{t-}^0=x$, $X_T^0=\xi$, $X_s^0=\gamma_s^{-\frac12} (u_s^0 - H_s^0)$, $s \in [t,T)$, where $H^0$ is given by \eqref{eq:SDEH0}. 
	By \Cref{lem:getXfromu} it holds that $X^0 \in \cA_t^{pm}(x,d)$. We furthermore obtain from \Cref{lem:getXfromu} that, for the associated deviation,  $D^0=\gamma X^0 + \gamma^{\frac12}H^0$. By definition of $X^0$, this yields $\gamma_s^{-\frac12}D_s^0=u_s^0$, $s\in[t,T)$. 
	It follows that 
	\begin{equation*}
		\begin{split}
			\md(X^n,X^0) 
			& = \left( E\left[ \int_t^T (\gamma_s^{-\frac12} D^n_s - \gamma_s^{-\frac12} D_s^0)^2 ds \right] \right)^{\frac12} 
			= \lVert \gamma^{-\frac12} D^n - u^0 \rVert_{\cL_t^2} ,
		\end{split}
	\end{equation*}
	and hence $\lim_{n\to\infty} \md(X^n,X^0)=0$.
\end{proof}

\begin{proof}[Proof of \Cref{lem:givenXpmfindustcostfcteq}]
	By definition of $u$ we have that $u$ is progressively measurable and, due to assumption (A1), satisfies $E[\int_t^T u_s^2 ds ]<\infty$; hence, $u \in \cL_t^2$.
	
	Let $\ssfi_s=\gamma_s^{-\frac12} D_s - \gamma_s^{\frac12} X_s$, $s \in [t,T]$, be the scaled hidden deviation \eqref{eq:scaledhiddendevdef} associated to $X$. 
	We can substitute $u=\gamma^{-\frac12} D$ in the cost functional~\eqref{eq:quadr_pmj} and also in the dynamics~\eqref{eq:dynscaledhiddendev} of $\ssfi$. 
	Observe that $\ssfi$ follows the same dynamics as the state process $\ssfiu$ associated to $u$ (see \eqref{eq:controlledprocdyn}), 
	and that 
	$\ssfi_t = \frac{d}{\sqrt{\gamma_t}}-\sqrt{\gamma_t}x = \ssfiu_t$. 
	Therefore, $\ssfi$ and $\ssfiu$ coincide, which completes the proof.
\end{proof}

\begin{proof}[Proof of \Cref{propo:givenugetX}]
	It follows from \Cref{lem:getXfromu} that $X \in \cA_t^{pm}(x,d)$. 
	Moreover, we have from \Cref{lem:getXfromu} that the associated deviation satisfies 
	$D=\gamma X + \gamma^{\frac12} \ssfiu$, i.e., $D_s = \gamma_s^{\frac12} u_s$, $s \in [t,T)$, 
	and $\ssfiu$ is the scaled hidden deviation of $X$. 
	It thus holds that $\pmJ_t(x,d,X)$ is given by~\eqref{eq:quadr_pmj}. 
	In the definition \eqref{eq:defcostfct2} of $\LQJ$, we may 
	replace 
	$u$ under the integrals with respect to the Lebesgue measure 
	by $\gamma^{-\frac12} D$. 
	This shows that 
	$\pmJ_t(x,d,X)=\LQJ_t( \frac{d}{\sqrt{\gamma_t}} - \sqrt{\gamma_t} x, u ) - \frac{d^2}{2\gamma_t}$.
\end{proof}

\begin{proof}[Proof of \Cref{lem:trafouvsuhat}]
	(i) 
	We have that $\hat u$ is progressively measurable. Furthermore, the facts that $E[\int_t^T u_s^2 ds]<\infty$, $E[\sup_{s \in [t,T]} \ssfiu_s^2]<\infty$, $E[\int_0^T \gamma_s \zeta_s^2 ds] <\infty$, and~\eqref{eq:cond_no_ct} 
	imply that $E[\int_t^T \hat u_s^2 ds]<\infty$. Hence, $\hat u \in \cL_t^2$. 
	Substituting 
	$u_s = \hat u_s + \frac{\lambda_s}{\lambda_s+\kappa_s} (\ssfiu_s + \sqrt{\gamma_s} \zeta_s )$, $s\in[t,T]$, 
	in \eqref{eq:controlledprocdyn} leads to  \eqref{eq:controlledprocdynhat}. 
	For the cost functional, observe that 
	\begin{equation}\label{eq:addingthesquareincostfct}
		\begin{split}
			& \frac12 \Big( 2\lambda_s ( \ssfiu_s + \sqrt{\gamma_s} \zeta_s )^2 - 4\lambda_s ( \ssfiu_s + \sqrt{\gamma_s} \zeta_s ) u_s \Big) + (\kappa_s+\lambda_s) u_s^2 \\
			& = \lambda_s ( \ssfiu_s + \sqrt{\gamma_s} \zeta_s )^2 - (\lambda_s+\kappa_s) \frac{\lambda_s^2}{\left(\lambda_s+\kappa_s\right)^2} ( \ssfiu_s + \sqrt{\gamma_s} \zeta_s )^2 \\
			& \quad + (\lambda_s +\kappa_s) \left(u_s - \frac{\lambda_s}{\lambda_s+\kappa_s} ( \ssfiu_s + \sqrt{\gamma_s} \zeta_s ) \right)^2 \\
			& = \frac{\lambda_s \kappa_s}{\lambda_s+\kappa_s} \left( \ssfihat_s + \sqrt{\gamma_s} \zeta_s \right)^2 + (\lambda_s+\kappa_s) \hat u_s^2 ,
			\quad s\in[t,T].
		\end{split}
	\end{equation}
	
	(ii) 
	Note that \eqref{eq:controlledprocdynhat} is an SDE that is linear in $\ssfihat$, $\hat u$, and $\sqrt{\gamma}\zeta$. Furthermore, boundedness of $\rho,\mu,\sigma,\eta,\rcor$ 
	and~\eqref{eq:cond_no_ct} 
	imply that the coefficients of the SDE are bounded. 
	Since moreover $E[ \int_t^T (\hat u_s)^2 + \gamma_s \zeta_s^2 ds ] <\infty$ and $\ssfihat_t$ is square integrable, we know that $E[ \sup_{s \in [t,T]} \ssfihat_s^2 ]<\infty$  (see, e.g., \cite[Theorem~3.2.2 and Theorem~3.3.1]{zhang}).	
	We can thus argue similar to (i) that $u \in \cL_t^2$. 
	A substitution of $\hat u$ in \eqref{eq:controlledprocdynhat} yields \eqref{eq:controlledprocdyn}. A reverse version of the argument in \eqref{eq:addingthesquareincostfct} proves equality of the cost functionals.
\end{proof}

\bigskip 
\textbf{Acknowledgement:}
We thank
Dirk Becherer, Tiziano De Angelis, Miryana Grigorova, Martin Herdegen, and Yuri Kabanov 
for inspiring discussions.
We are grateful
to the associate editor and two anonymous referees
for constructive comments and suggestions that helped us improve the manuscript.

\bibliographystyle{abbrv}
\bibliography{literature}

\end{document}